\documentclass[a4paper]{article}
\usepackage{geometry}
\usepackage{mathtools} %
\usepackage{amsmath, amsthm, amssymb}%
\usepackage{bm}
\usepackage{enumitem}
\usepackage{color}
\usepackage{hyperref}\hypersetup{colorlinks,linkcolor=red,anchorcolor=blue,citecolor=blue}
\usepackage{mathrsfs}
\usepackage{tikz}
\usetikzlibrary{intersections}

\usepackage{graphicx} %

\usepackage{mathrsfs}  
\usepackage{booktabs}

\usepackage{indentfirst}

\allowdisplaybreaks
\makeatletter 

\let\epsilon\varepsilon

\usepackage{tcolorbox}

\numberwithin{equation}{section}

\theoremstyle{definition}

\newtheorem{assumption}{Assumption}[section]

\theoremstyle{plain}

\newtheorem{theorem}{Theorem}[section]

\newtheorem{lemma}[theorem]{Lemma}
\newtheorem{cor}[theorem]{Corollary}

\newcommand{\bbR} {\mathbb{R}}
\newcommand{\bbP} {\mathbb{P}}

\newcommand{\R}{\mathbb{R}}
\newcommand{\N}{\mathbb{N}}

\newcommand{\calP}{\mathcal{P}}

\newcommand{\QOT}{\texttt{QOT}}
\newcommand{\OT}{\texttt{OT}}
\newcommand{\rmd}{\mathrm{d}}

\DeclareMathOperator{\spt}{spt}
\DeclareMathOperator{\interior}{int}
\DeclareMathOperator{\Diam}{diam}
\DeclareMathOperator{\dist}{dist}
\DeclareMathOperator{\gr}{gr}

\usepackage{kantlipsum} %

\begin{document}
\title{Sparsity of Quadratically Regularized Optimal Transport: Bounds on concentration and bias}
\author{Johannes Wiesel\footnote{Department of Mathematics, Carnegie Mellon University, e-mail:\,wiesel@cmu.edu. JW acknowledges support by NSF Grant DMS-2345556. JW would like to thank the Sydney Mathematical Research Institute for their hospitality.} \and Xingyu Xu\footnote{Department of Electrical and Computer Engineering, Carnegie Mellon University, e-mail:\,xingyuxu@andrew.cmu.edu. XX acknowledges support by NSF Grant ECCS-2126634.} }

\maketitle

\begin{abstract}
We study the quadratically regularized optimal transport (QOT) problem for quadratic cost and compactly supported marginals $\mu$ and $\nu$. It has been empirically observed that the optimal coupling $\pi_\epsilon$ for the QOT problem has sparse support for small regularization parameter $\epsilon>0.$ In this article we provide the first quantitative description of this phenomenon in general dimension: we derive bounds on the size and on the location of the support of $\pi_\epsilon$ compared to the Monge coupling. Our analysis is based on pointwise bounds on the density of $\pi_\epsilon$ together with Minty's trick, which provides a quadratic detachment from the optimal transport duality gap. In the self-transport setting $\mu=\nu$ we obtain optimal rates of order $\epsilon^{\frac{1}{2+d}}.$
\end{abstract}


\section{Introduction and main results}

The optimal transport problem with quadratic cost is 
defined as
\begin{equation*}
C(\mu,\nu):= \inf_{\pi \in \Pi(\mu, \nu)}\int \frac{1}{2} \|x-y\|^2 \,\pi(\rmd x, \rmd y).\quad 
\eqno{(\OT)}
\end{equation*}
Here  $\mu,\nu$ are probability measures on $\R^d$ and $\Pi(\mu,\nu)$ denotes the set of couplings of $\mu$ and $\nu$. In this article we study a regularized version of $C(\mu,\nu)$ given by 
\begin{equation*}
  \inf_{\pi \in \Pi(\mu, \nu)}\int \frac{1}{2}\|x-y\|^2 \, \rmd\pi + \frac{\varepsilon}{2} \left\| \frac{\rmd \pi}{\rmd \bbP} \right\|^2_{L^2(\bbP)},  \quad \eqno{(\QOT)}
\end{equation*}
where $\bbP:=\mu \otimes \nu$, the product measure of $\mu$ and $\nu$, and $\epsilon>0$. $(\QOT)$ is called a \emph{quadratically regularized optimal transport problem}. If  $(\QOT)$ is finite, then the unique optimizer $\pi_\epsilon$ has a density with respect to $\bbP$, and it is well known (see e.g.,\,\cite{nutz2024quadratically}), that 
\begin{align}\label{eq:optimal_density}
 \frac{\rmd \pi_{\varepsilon}}{\rmd \bbP}(x,y) := \frac{1}{\varepsilon} [ f_{\varepsilon}(x)+ g_{\varepsilon}(y) - c(x,y)]_+
\end{align}
holds for some continuous functions $f_\epsilon, g_\epsilon:\R^d\to \R$ and $c(x,y):=\|x-y\|^2/2$. We denote the support of $\pi_\epsilon$ by
\begin{equation}
\label{eqn: def spt}
\spt \pi_{\varepsilon} := \overline{\left\{ (x, y) \in \spt \bbP: f_{\varepsilon}(x) + g_{\varepsilon}(y) > c(x,y) \right\}}.
\end{equation}
On the other hand, if we assume that $\mu$ is absolutely continuous with respect to the Lebesgue measure, then Brenier's theorem yields existence of a unique optimizer $\pi_\star$ for $(\OT)$, and $$\spt \pi_\star \subseteq \{ (x, \nabla \varphi(x)): x\in \R^d\}$$ for the convex Brenier potential $\varphi:\R^d \to \R$.

It is natural to expect that $\pi_\epsilon\approx \pi_\star$ for small $\epsilon$; in particular $\spt \pi_\epsilon$ is sparse. Indeed, \cite[Theorem 4.1]{nutz2024quadratically} recently established Hausdorff convergence of $\spt \pi_\epsilon$ to $\spt \pi_\star$ for $\epsilon\to 0.$ The main goal of this article is to quantify this convergence.  More precisely we answer the following question:

\begin{tcolorbox}
How fast does $\spt \pi_{\varepsilon}$ concentrate around $\spt \pi_{\star} $, or in other words: can we find upper bounds on the size of $\spt \pi_{\varepsilon}$ and the distance of $\spt \pi_{\varepsilon}$ to $\spt \pi_{\star}$?
\end{tcolorbox}

As anticipated, our two main results comprise a pointwise bound on the concentration of $\spt \pi_\epsilon$ and a pointwise bound on its bias (i.e. its distance to $\spt \pi_{\star}$). These are the first quantitative bounds on the support of $ \pi_\epsilon$ in general dimension. They can be informally summarized as follows: define the \emph{$\epsilon$--spread of $\mu$} by
\begin{align*}
\delta(\varepsilon) \coloneqq \inf\{r>0 : r \cdot \rho(r) > \varepsilon\}, \quad 
  \rho(r) \coloneqq \inf_{x \in \spt\mu} \mu(B(x, r)).
\end{align*}
Intuitively, $\delta(\epsilon)$ captures how uniformly the probability mass of $\mu$ is spread over its support. In particular, for $\mu=\text{Unif}(B_{1}(0))$ we have $\delta(\epsilon)\asymp \epsilon^{\frac{1}{1+d}}.$ With this definition at hand, Theorem \ref{thm: concentration} in Section \ref{sec:general} states the \emph{concentration bound}
\begin{align}\label{eq:intro1}
\sup_{(x,y)\in \spt\pi_\epsilon} \inf_{x'\in \R^d} \| (x,y)-(x', \nabla\psi_\epsilon(x'))\|\le C\sqrt{\delta(\epsilon)}.
\end{align}
Here $\psi_\epsilon$ is a convex continuously differentiable function, which is explicitly constructed in Section \ref{sec:prep} below. Paraphrasing \eqref{eq:intro1}, $\spt \pi_\epsilon$ is concentrated around the graph of the gradient of a convex function. Under the assumption that $\spt \mu$ is star-shaped, this result can be improved: Theorem \ref{thm: bdry ub} states the \emph{bias bound}
\begin{align}\label{eq:intro2}
   \sup_{(x,y)\in \spt\pi_\epsilon} \| y - \nabla \varphi(x) \| \le C (L+1)^{3/2} 
\sqrt[4]{\delta\Big(\frac{\delta(\varepsilon)}{L+1}\Big)}, 
\end{align}
where $L$ is the Lipschitz-constant of $\nabla \varphi$.
These rates hold under mild assumptions on $\mu, \nu$ and explicitly describe the interplay between $\mu,\nu$ and the concentration of $\spt\pi_\epsilon$ through the $\epsilon$--spread $\delta(\epsilon)$ and the Lipschitz constant $L$. As noted above, we typically expect $\delta(\epsilon)\asymp \epsilon^{\frac{1}{1+d}}$ if $\mu$ is comparable to $\text{Unif}(B_{1}(0))$.

In the self-transport setting $\mu=\nu$ our results can be significantly improved. In fact, Theorems \ref{thm: sym ub} and \ref{thm: sym lb} in Section \ref{sec:self-transport} state the sharp upper and lower bounds
$$ \sup_{(x, y) \in \spt\pi_\epsilon} \| y - \nabla \varphi(x) \|
\asymp \delta_{\mathsf{ST}}^{1/2}(\epsilon),$$
where
\[
\delta_{\mathsf{ST}}(\varepsilon) \coloneqq \inf\{r>0 : r \cdot \rho(\sqrt r) > \varepsilon\}, 
\quad 
\rho(r) \coloneqq \inf_{x \in \spt\mu} \mu(B(x, r))
\]
is the \emph{improved $\epsilon$--spread}. Taking again $\mu$ comparable to $\text{Unif}(B_{1}(0))$ yields \emph{optimal rates} of order $\epsilon^{\frac{1}{2+d}}$ (see Corollary \ref{rem:rate} for an exact statement).

\definecolor{azure}{rgb}{0.0, 0.5, 1.0}
\definecolor{darkpastelgreen}{rgb}{0.01, 0.75, 0.24}

\begin{figure}
\centering
\begin{tikzpicture}
    \draw[->] (0,0) -- (6.7,0) node[right] {$x$};
    
    \draw[->] (0,0) -- (0, 5.2) node[above] {$y$};
    
    \draw[color=azure, line width=0.25mm] (0,0) .. controls (1.0, 1.75) and (4, 1.85) .. (5.5, 3) node[right]{$x\mapsto \nabla\psi_\epsilon(x)$};

    \draw[color=azure, dashed, line width=0.25mm] (0,0+0.3) .. controls (1.0, 1.75+0.3) and (4, 1.85+0.3) .. (5.5, 3+0.3);

    \draw[color=azure, dashed, line width=0.25mm] (0.3, 0) .. controls (1.0, 1.75-0.3) and (4, 1.85-0.3) .. (5.5, 3-0.3);

    \draw[color=black, line width=0.25mm] (0,0) .. controls (1.75, 0.25) and (2.5, 3.55) .. (5.2, 4) node[right]{$x\mapsto \nabla\varphi(x)$};

    \draw[color=darkpastelgreen, dashed, line width=0.25mm] (0,0+2) .. controls (1.25, 3) and (1.75, 5) .. (4.5, 5.5);

    \draw[color=darkpastelgreen, dashed, line width=0.25mm] (0+2.5,0) .. controls (1.75+2.25, 0.25) and (2.5+1.5, 2.15) .. (4.75+1, 2.55);

    \newcommand{\ubpath}{(0,0+0.15) .. controls (1.0, 1.75+0.15) and (4, 1.85+0.15) .. (5.5, 3+0.15)}
    \newcommand{\lbpath}{(0.15,0) ..  controls (1.0, 1.75-0.15)  and (4, 1.85-0.15)  ..  (5.5, 3-0.15)}

    \begin{scope}
        \clip \ubpath -- (6.5, 3+0.15) -- (6.5, 0) -- (0, 0) -- cycle;
        \fill[azure, opacity=0.4] (0,0) -- \lbpath -- (5.5,4) -- (0, 4) -- cycle;
    \end{scope}
\end{tikzpicture}
\caption{Plot of $\spt \pi_\star=\{(x,\nabla\varphi(x))\}$ (black), $\spt \pi_\epsilon$ (blue) and $x\mapsto \nabla \psi_\epsilon(x)$ (blue) with bounds of order $\sqrt{\delta(\epsilon)}$ (dashed blue) and  $\delta(\delta(\epsilon)/(L+1))$ (dashed green).}
\label{fig:1}
\end{figure}
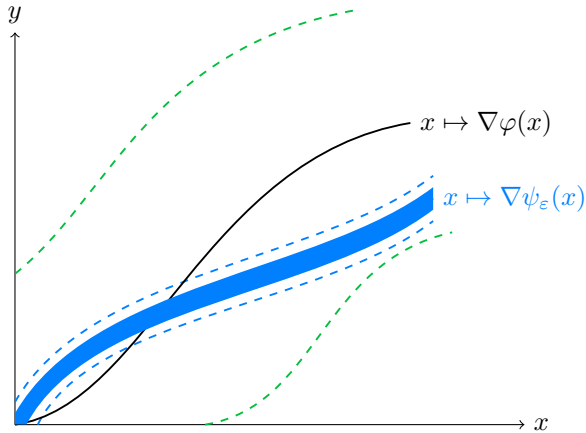

Figure \ref{fig:1} gives a pictorial description of our results: while the inclusion $\spt \pi_\star\subseteq \spt\pi_\epsilon$ remains an open research problem (see \cite{gonzalez2024monotonicity} for a recent counterexample in a discrete setting), our results give pointwise upper bounds on the distance of $\spt \pi_\epsilon$ and $\spt \pi_\star,$ which naturally inform the choice of $\epsilon$ for a specific pair of measures $(\mu,\nu)$.

\subsection{Related literature}

Quadratic regularization of optimal transport problems goes back at least to \cite{blondel2018smooth, essid2018quadratically} for finitely supported probability measures. For continuous measures $\mu,\nu$, first results are obtained in \cite{lorenz2022orlicz, lorenz2021quadratically}. QOT is empirically known to yield optimal couplings $\pi_\epsilon$ with sparse support. This is in stark contrast to the entropically regularized optimal transport (EOT) problem (see e.g.,\,\cite{cuturi2013sinkhorn, genevay2019sample, mena2019statistical, ghosal2022stability, bayraktar2022stability, nutz2021introduction, peyre2019computational} and the references therein), for which optimal couplings always have full support, even for very small $\epsilon>0.$ In some instances, such a full support property can be undesirable, and can lead to numerical instabilities, see e.g.,\,\cite{blondel2018smooth} for a discussion in the context of image processing. As \eqref{eq:optimal_density} is of linear growth, it often allows for significantly smaller regularization parameters $\epsilon>0$ compared to EOT (for which the optimal density grows exponentially fast). Recent applications of optimal transport with quadratic regularization include manifold learning \cite{zhang2023manifold} and learning of Wasserstein barycenters \cite{li2020continuous}.

The aim of this paper is to quantify the convergence of $(\QOT)$ to the $(\OT)$ problem for $\epsilon\to 0$. One of the first papers addressing this issue is \cite{lorenz2022orlicz}, who establish a Gamma convergence result. More recently, 
\cite{eckstein2024convergence} give quantitative rates for the optimal transport costs: they show that 
\begin{align*}
    C^{-1} \epsilon^{\frac{2}{2+d}} \le C_\epsilon(\mu,\nu)-C(\mu,\nu)\le C\epsilon^{\frac{2}{2+d}},
\end{align*}
while \cite{garriz2024infinitesimal} derive exact first-order asymptotics for $[C_\epsilon(\mu,\nu)-C(\mu,\nu)]/\epsilon^{\frac{2}{2+d}}.$ On the other hand, \cite{nutz2024quadratically} proves  convergence of $f_\epsilon$, $g_\epsilon$ to the Kantorovich potentials $f_\star, g_\star$.
Taking these results into account, it is natural to conjecture that optimal rates for $\|y-\nabla \varphi(x)\|$ are of order $\epsilon^{\frac{1}{d+2}}$ for all $(x,y)\in \spt\pi_\epsilon.$ However, to the best of the authors' knowledge, no quantitative results are available in the literature, and this paper is the first to offer an approach that fills this important gap.

We prove our main results by first approximating $\|\cdot\|^2/2-f_\epsilon, \|\cdot\|^2/2 - g_\epsilon$ by the convex function $\psi_\epsilon$ and its convex conjugate $\psi_\epsilon^*$, and then use Minty's trick \cite{minty1962monotone} to derive a quadratic detachment of the duality gap $\psi_\epsilon(x)+\psi_\epsilon^*(y)-\langle x,y\rangle$ from the set $\{(x,y)\in \R^d\times \R^d: \psi_\epsilon(x)+\psi_\epsilon^*(y)-\langle x,y\rangle=0\}.$ In short, Minty's trick says that the graph of a monotone operator is a $1$-Lipschitz graph in a rotated chart (see also Section \ref{sec:prep} for a more detailed discussion). Applying it in the context of optimal transport with quadratic cost is not new, and goes back at least to \cite{alberti1999geometrical}, see also \cite{mccann2012rectifiability, carlier2023convergence}.

While completing the first version of this paper, we learned about concurrent and ongoing research by Alberto Gonz\'alez-Sanz and Marcel Nutz who kindly shared their results. Their work has the same goal of bounding the size of the support of $\pi_\epsilon$ and the distance to the Monge map $\nabla \varphi $ for small $\epsilon$, but focuses on the one-dimensional setting $d=1$. Assuming that $\mu,\nu $ have continuous densities bounded from above and below, they prove that the thickness of the image of $x$ under (\QOT{}) is of order $\epsilon^{1/3}$, uniformly in $x\in \spt\mu$. Furthermore, they show that the Hausdorff distance between the images of $x$ under (\OT{}) and (\QOT{}) is of the same order $\epsilon^{1/3}$ (in an $L^2$ sense). Their analysis is based on the study of the shape of the potentials $f_\epsilon, g_\epsilon$. In particular they show that for small $\epsilon$, $\|\cdot\|^2/2-f_\epsilon$ is strongly convex, uniformly in $\epsilon.$ Such a uniform regularity result is naturally much harder to obtain in the multivariate case, and crucially depends on the assumption that $\mu,\nu$ have densities bounded from above and below. In that sense, our results are more general than theirs.
 
\subsection{Notation}

We work on $(\R^d, \|\cdot\|)$, where $d\in \N$. We write $\langle \cdot, \cdot \rangle$ for the scalar product on $\R^d.$ For a set $A\subseteq \R^d$ we use $\partial A$ for the topological boundary of $A$ and write $\mathrm{int} A $ for the interior of $A$. We also define $B(x, \epsilon)$ as the open ball of radius $\epsilon$ around a point $x$ and set $d(x,A):= \inf_{y\in A} \|x-y\|$ as well as $\Diam(A)=\sup_{x,y\in A} \|x-y\|.$ We denote by $\mathbf{1}_A$ the indicator function of the set $A$ and set $x_+:= \max(0,x)$ for $x\in \R.$ The notation $\text{Leb}$ is reserved for the Lebesgue measure on $\R^d$, while we write $f_\sharp\mu$ for the push-forward measure of a Borel probability measure $\mu\in \mathcal{P}(\R^d)$ under a measurable function $f:\R^d\to \R.$
Throughout this note we denote constants by $C$, which might change from line to line. The notation $x\asymp y$ means that there exist constants $c,C>0$ such that $cx\le y\le Cx$ holds. We sometimes write $x\lesssim y$ (resp. $x\gtrsim y$), meaning that there exists a constant $C>0$ such that $x\le Cy$ (resp. a constant $c>0$ such that $x\ge cy$).

To shorten notation we define $c(x,y):= \|x-y\|^2/2.$ If a function $f:\R^d\to \R$ is convex, then we denote by $\partial f(x):= \{z\in \R^d: f(y)-f(x)\ge \langle z, y-x\rangle \, \forall x,y \in \R^d\}$ the subdifferential of $f$.  We use the same notation for the superdifferential of $f$ if $f$ is concave and write $$\partial_x f(x,y):= \{z\in \R^d: f(x',y)-f(x,y)\ge \langle z, x'-x\rangle \quad  \forall x,x' \in \R^d\}$$ for the $x$-sub/superdifferential of a convex/concave function $f:\R^d\times \R^d\to \R.$ If a function $f:\R^d \to \R$ is differentiable in a point $x\in \R^d$, then we write $\nabla f(x)$ for its derivative in $x$. The convex conjugate of a function $f:\R^d\to \R$ is given by $f^*(y):=\sup_{x\in \R^d} [\langle x,y\rangle -f(x)].$ The graph of a function $f:\R^d\to \R^d$ is denoted by $\gr f\coloneqq \{(x, f(x)): x \in \bbR^d\}$.

We quantify sparsity of $\spt\pi_\epsilon$ by its distance to $\spt \pi_\star$. For this it is useful to define the \emph{asymmetric Hausdorff distance} between two sets $A, B \subseteq \bbR^d\times \bbR^d$ as
\begin{align} \label{eq:hausdorff}
\dist(A; B) \coloneqq \inf\{\epsilon > 0: A \subseteq B + B(0, \epsilon)\} = \sup_{x \in A} \inf_{y \in B} \|x - y\|,
\end{align} 
where $+$ denotes Minkowski sum of sets and $\|\cdot\|$ is the product norm (by a slight abuse of notation). If $\dist(\spt\pi_\epsilon; \spt\pi_\star)$ is small, this indicates that $\spt\pi_\epsilon$ is contained in a ``sausage" around $\spt\pi_\star$; in other words $\spt\pi_\epsilon$ is sparse.

\subsection{Overview}
The remainder of this article is organized as follows. We give an overview of the $(\QOT)$  problem and some preliminary results in Section \ref{sec:preliminary}. We then derive matching upper and lower bounds for $\dist(\spt\pi_\epsilon; \spt\pi_\star)$ in the self-transport case $\mu=\nu$ in Section \ref{sec:self-transport}. Sections \ref{sec:prep} and  \ref{sec:general}
are devoted to the general case $\mu\neq \nu$. We first show that dual potentials are approximately conjugate in Section \ref{sec:prep} and then use Minty's trick to turn bounds on the dual gap into bounds on the support of $\spt \pi_\epsilon.$ 

\section{Setting and preliminary results}\label{sec:preliminary}

\subsection{Assumptions on $\mu,\nu$}
We work with compactly supported probability measures $\mu$ and $\nu$ on $\R^d$. More specifically we make the following standing assumption on $\mu$.
\begin{assumption}
The probability measure $\mu \in \calP(\bbR^d)$ satisfies $\mu \ll \text{Leb}$, $\spt\mu \subseteq B(0, 1)$ and $\interior\spt\mu$ is connected.
\end{assumption}
This assumption is mostly needed to guarantee existence of the Brenier potential $\varphi: \bbR^d \to \bbR$, i.e. a convex function satisfying
\[\nu = (\nabla \varphi)_\sharp \mu.\] 
We call $\nabla \varphi$ the Monge map for the quadratic cost
\[c(x, y) = \frac12 \|x-y\|^2.\]
To simplify notation, we also make the following standing assumption on $\nu.$
\begin{assumption}\label{ass:nu}
The probability measure $\nu \in \calP(\bbR^d)$ satisfies $\spt\nu \subseteq B(0,1)$.
\end{assumption}
Note that Assumption \ref{ass:nu} can equivalently be stated as $\|\nabla \varphi\| \le 1$ almost surely in $\mu$.

\subsection{Preliminary results}

Recall that the optimal transference plan for $(\OT)$ is given by $\pi_\star = (\operatorname{id}, \nabla\varphi)_\sharp \mu$, and the Kantorovich potentials are
\[
f_\star(x) := \frac12 \|x\|^2 - \varphi(x),
\quad g_\star(y) := \frac12 \|y\|^2 - \varphi^*(y).
\]
It is well-known (see e.g.,\,\cite{nutz2024quadratically}), that the dual problem of (\QOT) is
\begin{align}\label{eq:dual}
  \sup_{f_{\varepsilon} \in L^1(\mu), g_{\varepsilon} \in L^1(\nu)} \int f_{\varepsilon} \,\rmd \mu + \int g_{\varepsilon} \,\rmd \nu - \frac{1}{2\varepsilon} \int [f_{\varepsilon} \oplus g_{\varepsilon} - c ]_+^2 \,\rmd \bbP.
\end{align}
The continuous functions $f_\epsilon,g_\epsilon$ from \eqref{eq:optimal_density} in the Introduction are in fact solutions of \eqref{eq:dual}; we thus call them the \emph{optimal dual potentials} throughout this note. Up to constant shifts, they are uniquely defined via the set of equations
\begin{align}
  \int \big[f_{\varepsilon}(x) + g_{\varepsilon}(y) - c(x,y)\big]_+ \,\mu(\rmd x) &= \varepsilon 
  \qquad \forall y\in\R^d, \label{eq:dual1}\\
  \int \big[f_{\varepsilon}(x) + g_{\varepsilon}(y) - c(x,y)\big]_+ \,\nu(\rmd y) &= \varepsilon 
  \qquad \forall x\in \R^d. \label{eq:dual2}
\end{align}
A direct consequence of \eqref{eq:dual1}-\eqref{eq:dual2} is the following:
\begin{cor}
We have
\begin{align}
    \sup_{x \in \spt \mu} \big[ f_{\varepsilon}(x) + g_{\varepsilon}(y) - c(x, y) \big] \ge \varepsilon 
    \qquad & \forall y \in \R^d,
    \label{eqn: density lower bound x}
    \\
    \sup_{y \in \spt \nu} \big[ f_{\varepsilon}(x) + g_{\varepsilon}(y) - c(x, y) \big] \ge \varepsilon 
    \qquad & \forall x \in \R^d.
    \label{eqn: density lower bound y}
  \end{align}
\end{cor}
We can also derive the following concavity property from  \eqref{eq:dual1}-\eqref{eq:dual2}, which will be important for the self-transport case $\mu=\nu$ discussed in Section \ref{sec:self-transport}.
\begin{lemma}\label{lem:concave}
The functions $x\mapsto f_\epsilon(x) -\|x\|^2/2$,  $y\mapsto g_\epsilon(y) -\|y\|^2/2$ are concave.
\end{lemma}
\begin{proof}
We only show that $x \mapsto f_\epsilon(x) -\|x\|^2/2 $ is concave. For this it is enough to realize that for any $x, \tilde x\in \R^d$ and any $\lambda\in [0,1]$ we have
\begin{align*}
&\int [\lambda (f_\epsilon(x)-\|x\|^2/2)+(1-\lambda) (f_\epsilon(\tilde x)-\|\tilde x\|^2/2) +g_\epsilon(y) -\|y\|^2/2 +\langle \lambda x +(1-\lambda) \tilde x, y\rangle  ]_+\,\nu(dy)\\
&\le \lambda \int [ f_\epsilon(x)-\|x\|^2/2+g_\epsilon(y)-\|y\|^2/2 +  \langle x,y\rangle ]_+\,\nu(dy)\\
&\qquad
 + (1-\lambda) \int [ f_\epsilon(\tilde x)-\|\tilde x\|^2/2+g_\epsilon(y) -\|y\|^2/2 + \langle \tilde x, y\rangle ]_+\,\nu(dy) \\
&=\lambda \int [f_\epsilon(x)+g_\epsilon(y)-c(x,y)]_+\, \nu(dy)+ (1-\lambda) \int [f_\epsilon(\tilde x)+g_\epsilon(y)-c(\tilde x,y)]_+\, \nu(dy) = \epsilon,
\end{align*}
where we used \eqref{eq:dual2} for the last inequality.
By monotonicity of $x\mapsto (x+C)_+$ we conclude again by \eqref{eq:dual2} that we must have
\begin{align*}
 \lambda (f_\epsilon(x)-\|x\|^2/2)+(1-\lambda) (f_\epsilon(\tilde x)-\|\tilde x\|^2/2) \le f_\epsilon(\lambda x +(1-\lambda) \tilde x)-\|\lambda x +(1-\lambda) \tilde x\|^2/2.   
\end{align*}
This proves the claim.
\end{proof}

\section{Self-transport: matching upper and lower bounds}\label{sec:self-transport}

As a warm-up exercise we first consider a specific instance of $C_\epsilon(\mu,\nu)$, namely the self-transport setting $\mu=\nu$. This setting is considerably easier to handle than the general case and allows for matching upper and lower bounds of  $\dist(\spt \pi_\epsilon; \spt\pi_\star)$ in general dimension. 

To recap, $\mu=\nu$ implies $\varphi(x) = \|x\|^2/2$,  $\nabla \varphi = \operatorname{id}$, and $\spt\pi_\star = \{(x, x): x \in \spt\mu\}$. In particular $ \dist(\spt \pi_\epsilon; \spt\pi_\star)\le  \sup_{(x,y) \in \spt \pi_\epsilon} \|x-y\|.$
Throughout the analysis we will make frequent use of  \cite[Theorem 2.2.(iv)]{nutz2024quadratically}, stating that the dual potentials can be chosen to satisfy $f_{\varepsilon} = g_{\varepsilon}$ in this case.

\subsection{Main results}

We now state our main results for the self-transport setting. As mentioned in the Introduction, we find matching upper and lower bounds in this case (see Theorems \ref{thm: sym ub}, \ref{thm: sym lb}).

\begin{theorem}[Upper bound, self-transport]
  \label{thm: sym ub}
  There exists a universal constant $C_{\mathsf{ub}}>0$ such that 
  \[
  \sup_{(x, y) \in \spt\pi_\epsilon} \|x - y\| \le C_{\mathsf{ub}} \min\big( \delta_{\mathsf{ST}}^{1/2}(\varepsilon), \Diam(\spt\mu) \big),
  \]
  where the improved $\epsilon$--spread $\delta_{\mathsf{ST}}(\varepsilon)$ is defined by 
  \[
  \delta_{\mathsf{ST}}(\varepsilon) \coloneqq \inf\{r>0 : r \cdot \rho(\sqrt r) > \varepsilon\}, 
  \quad 
  \rho(r) \coloneqq \inf_{x \in \spt\mu} \mu(B(x, r)).
  \]
  Consequently we have
  \[
  \dist(\spt \pi_\epsilon; \spt\pi_\star) \le C_{\mathsf{ub}} \min\big( \delta_{\mathsf{ST}}^{1/2}(\varepsilon), \Diam(\spt\mu) \big).
  \]
\end{theorem}

We list a few properties of $\delta_{\mathsf{ST}}$, which follow directly from the definition and will be frequently used in the proofs below.

\begin{lemma}
  We have the following:
  \begin{itemize}
      \item $\epsilon\mapsto \delta_{\mathsf{ST}}(\varepsilon)$ is monotone increasing.
      \item $\delta_{\mathsf{ST}}(\varepsilon) \ge \varepsilon$ for $\varepsilon>0$.
      \item $\delta_{\mathsf{ST}}(\varepsilon) = \varepsilon$ for $\varepsilon\ge 4$ since $\rho(r) = 1$ for $r \ge 2$. 
      \item $\delta_{\mathsf{ST}}(C\varepsilon) \le \max(C, 1) \delta_{\mathsf{ST}}(\varepsilon)$.
  \end{itemize}
\end{lemma}

If $\mu$ is comparable to the Lebsgue measure, we obtain the following corollary.

\begin{cor}[Explicit rate]\label{rem:rate}
Assume that $\mu$ has a density bounded away from zero, i.e.,\,$\mu \ge c \cdot \mathrm{Leb}|_{\spt \mu}$ holds for some constant $c>0$, and $\spt\mu$ has $L$-Lipschitz boundary. Then there exists a constant $C_{c, L}>0$ depending only on $c, L$ such that
\begin{align*}
\dist(\spt \pi_\epsilon; \spt\pi_\star) \le C_{c, L} \min\big(\epsilon^{\frac{1}{d+2}} , \Diam(\spt\mu) \big).
\end{align*}
\end{cor}

\begin{proof}
Under the assumptions on $\mu$ it follows that $\rho(r)^{1/d} \asymp_{c, L} r$, hence $\delta_{\mathsf{ST}}(\epsilon) \asymp_{c, L} \epsilon^{\frac{2}{d+2}}$. The claim now follows from Theorem~\ref{thm: sym ub}.  
\end{proof}

Theorem~\ref{thm: sym ub} is optimal, in the sense of the following matching lower bound.
\begin{theorem}[Lower bound, self-transport]
  \label{thm: sym lb}
  There exists a constant $C_{\mathsf{lb}}>0$ such that for any $\epsilon > 0$, we have $$\dist(\spt\pi_\epsilon; \spt\pi_\star) \ge C_{\mathsf{lb}} \min\big( \delta_{\mathsf{ST}}^{1/2}(\varepsilon), \Diam(\spt\mu) \big).$$ More precisely, there exists $(x, y) \in \spt \pi_{\varepsilon}$ satisfying $$\|x - y\| \ge \sqrt{2} C_{\mathsf{lb}} \min\big( \delta_{\mathsf{ST}}^{1/2}(\varepsilon), \Diam(\spt\mu) \big).$$
\end{theorem}

The remainder of this section is devoted to the proof of Theorems \ref{thm: sym ub} and \ref{thm: sym lb}.

\subsection{Proof of the upper bound}

We now detail the proof Theorem~\ref{thm: sym ub}, which will proceed via two lemmas. The first controls the distance between $\spt\pi_\epsilon$ and $\spt\pi_\star$ via a simple upper bound on the dual potential.

\begin{lemma}
\label{lem: sym ub by M}
Define $M := \sup_{x \in \spt\mu} f_\epsilon(x).$ Then we have 
\[
\spt \pi_{\varepsilon} \subseteq \{(x, y) \in \spt \bbP: \|x - y\|^2 \le 4M\}.
\]
\end{lemma}
\begin{proof}
Recall that $f_{\varepsilon} = g_{\varepsilon}$. 
It is evident that
  \[
  \sup_{(x, y) \in \spt\bbP} \big[ f_{\varepsilon}(x) + g_{\varepsilon}(y) \big] 
  = 2 \sup_{x \in \spt\mu} f_{\varepsilon}(x) = 2M.
  \]
Therefore
  \begin{align*}
  \spt \pi_{\varepsilon} = \overline{\{(x,y) \in \spt\bbP: f_{\varepsilon}(x) + f_{\varepsilon}(y) - c(x, y) > 0\}}
  & \subseteq \{(x,y) \in \spt\bbP: c(x, y) \le 2M\} 
  \\
  & = \{(x,y) \in \spt\bbP: \|x-y\|^2 \le 4M\}.
  \end{align*}
  This concludes the proof.
\end{proof}

By Lemma \ref{lem:concave} we know that the function
\begin{align*}
(x,y) \mapsto f_\epsilon(x) + g_\epsilon(y)-c(x,y) = f_\epsilon(x)-\|x\|^2/2 +g_\epsilon(y)-\|y\|^2 +\langle x,y\rangle
\end{align*}
is concave. The following lemma bounds the supergradient of $(x,y)\mapsto f_\epsilon(x) + g_\epsilon(y)-c(x,y)$ by the maximal spread of the support of $\spt\pi_\epsilon$.

\begin{lemma}[Gradient estimate, self-transport]
\label{lem: grad ub}
For $x\in \spt \mu$ define
\begin{align*}
\bar{y}(x) := \frac{\int_{\{y: (x, y)\in \spt\pi_\epsilon\}} y \, \nu(\rmd y)}{\nu(\{y: (x, y)\in \spt\pi_\epsilon\})}.
\end{align*}
Then $\bar{y}(x) -y \in \partial_x (f_\epsilon(x) + g_\epsilon(y) - c(x, y))$ and we have
\[
\sup_{(x, y) \in \spt \pi_\epsilon} \left\| \bar{y}(x) -y  \right\| 
 \le 2 \sup_{(x, y) \in \spt\pi_\epsilon} \|x - y\|.
\]
\end{lemma}

\begin{proof}
This can be proved by differentiating the \eqref{eq:dual2}, but to avoid differentiability issues, we take a different approach here.

Fix $(x_0, y_0) \in \spt\pi_\epsilon$. For any $x \in \spt\mu$, taking the subgradient of the convex function $x\mapsto x_+$ gives
\begin{align*}
&(f_\epsilon(x) + g_\epsilon(y) - c(x, y))_+ - (f_\epsilon(x_0) + g_\epsilon(y) - c(x_0, y))_+ 
\\
&\quad \ge \mathbf{1}_{\{f_\epsilon(x_0) + g_\epsilon(y) - c(x_0, y) \ge 0\}} \Big[ \big( f_\epsilon(x) + g_\epsilon(y) - c(x, y) \big) - \big( f_\epsilon(x_0) + g_\epsilon(y) - c(x_0, y) \big) \Big].
\end{align*}
 Integrating this inequality over $y$ and invoking \eqref{eq:dual2} for $x$ and $x_0$ we obtain
\[
\int_{\{y: (x_0, y)\in \spt\pi_\epsilon\}} \Big[ \big( f_\epsilon(x) - \frac12\|x\|^2 \big) -  \big( f_\epsilon(x_0) - \frac12\|x_0\|^2 \big) + \langle x - x_0, y\rangle \Big]\, \nu(\rmd y) \le 0.
\]
Using again \eqref{eq:dual2} we conclude $\nu(\{y: (x_0,y) \in \spt \pi_\epsilon\}) > 0$. Thus
\[
\big( f_\epsilon(x) - \frac12\|x\|^2 \big) - \big( f_\epsilon(x_0) - \frac12\|x_0\|^2 \big) \le \left\langle x - x_0, -\bar{y} \right\rangle, \quad \bar{y} = \frac{\int_{\{y: (x_0, y)\in \spt\pi_\epsilon\}} y \, \nu(\rmd y)}{\nu(\{y: (x_0, y)\in \spt\pi_\epsilon\})}.
\]
This implies that $-\bar{y}$ is a supergradient of $f_\epsilon(x) - \|x\|^2/2$ at $x_0$, hence
\[
y_0 - \bar{y} \in \partial _x(f_\epsilon(x_0) + g_\epsilon(y_0) - c(x_0, y_0)). 
\]
It now suffices to show $\|y_0 - \bar{y}\| \le 2\sup_{(x, y) \in \spt\pi_\epsilon} \|x - y\|$. This follows easily from 
\begin{align*}
\|y_0 - \bar{y}\| \le \frac{\int_{\{y: (x_0, y)\in \spt\pi_\epsilon\}} \| y_0 - y \| \,\nu(\rmd y)}{\nu(\{y: (x_0, y)\in \spt\pi_\epsilon\})}
& \le \sup_{y: (x_0, y)\in \spt\pi_\epsilon} \|y_0 - y\| 
\\
& \le  \| y_0-x_0\|+ \sup_{y: (x_0, y)\in \spt\pi_\epsilon} \|x_0 - y\|  \\
& \le 2\sup_{(x,y)\in \spt\pi_\epsilon} \|x - y\|,
\end{align*}
where the last line used $(x_0, y_0) \in \spt \pi_\epsilon$. This completes the proof.
\end{proof}

We are now in a position to give the proof of Theorem~\ref{thm: sym ub}.
\begin{proof}[Proof of Theorem~\ref{thm: sym ub}]
First note that $\|x-y\| \le \Diam(\spt\mu)$ holds for $(x,y) \in \spt\pi_\epsilon \subseteq \spt\mu \times \spt \mu$, thus it suffices to show $\|x-y\|^2 \lesssim \delta_{\mathsf{ST}}(\epsilon)$ for $(x, y) \in \spt\pi_\epsilon$.
Combining Lemma~\ref{lem: sym ub by M} with Lemma~\ref{lem: grad ub}, we readily obtain for all $(x, y) \in \spt\pi_\epsilon$ that
  \begin{equation}
  \label{eqn: grad ub by M}
  \| \bar{y}(x)-y \| \le 4\sqrt{M},
  \end{equation}
  where we recall that $M = \sup_{x \in \spt\mu} f_\epsilon(x)$. 
  Let $y_0 \in \spt\mu$ be such that $f_\epsilon(y_0) = g_\epsilon(y_0) = M$ (such $y_0$ exists since $f_\epsilon = g_\epsilon$ is continuous). 
  Define
  \[K_{M} := \{x \in \spt\mu: f_\epsilon(x) + g_\epsilon(y_0) - c(x, y_0) \ge M \}.\] 
  It is clear that $y_0 \in K_{M}$ as $c(y_0,y_0)=0$, and $(x, y_0) \in \spt \pi_\epsilon$ for $x \in K_{M}$ by \eqref{eqn: def spt}. We claim that 
  \begin{equation}
  \label{eqn: large spread by grad ub}
  B(y_0, \sqrt{M} / 4) \cap \spt\mu \subseteq K_{M} \subseteq \spt\pi_\epsilon. 
  \end{equation}
  In fact, the last inclusion follows from \eqref{eqn: def spt}, so we only need to prove the first inclusion. For this let $r$ be the supremum of all positive numbers such that $B(y_0, r) \cap \spt\mu \subseteq K_{M}$. 
  It suffices to show that $r \ge \sqrt{M}/4$, and for this purpose we suppose to the contrary that $r < \sqrt{M}/4$. The definition of $r$ implies that for any $\theta>0$, there exists $x_\theta \in \spt\mu$ such that $r \le \|x_\theta - y_0\| < r + \theta$ and $x_\theta \not\in K_{M}$. By compactness, one may find a sequence  $\theta_n \to 0$ such that $x_{\theta_n}$ converges to some point $x$. It is evident that $\|x - y_0\| = r$, hence $x \in \overline{B(y_0, r)} \cap \spt \mu \subseteq \overline{K_{M}} = K_{M}$ (note that $K_{M}$ is closed as it is the preimage of a closed set under a continuous map). But $x$ as the limit of  $x_{\theta_n} \not \in K_{M}$ satisfies $x \in \partial K_{M}$. Therefore, we have $f_\epsilon(x) + g_\epsilon(y_0) - c(x, y_0) = M$. Consequently $(x, y_0) \in \spt\pi_\epsilon$.
  Recalling~\eqref{eqn: grad ub by M} and using concavity of $x\mapsto f_\epsilon(x) + g_\epsilon(y) - c(x, y)$, we obtain
  \begin{align*}
  M = f_{\epsilon}(x) + g_{\epsilon}(y_0) - c(x, y_0)
  & \ge f_{\epsilon}(y_0) + g_{\epsilon}(y_0) - c(y_0, y_0) + \langle \bar{y}(x)-y,  x - y_0\rangle
  \nonumber \\
  & \ge 2M - \| \bar{y}(x)-y \| \|x-y_0\| 
  \nonumber \\
  & \ge 2M - 4\sqrt{M} \cdot r
  \nonumber\\
  & > 2M - 4\sqrt{M} \cdot \sqrt{M} / 4 = M,
  \end{align*}
  a contradiction. Thus $r \ge \sqrt{M} / 4$.
  This completes the proof that $B(y_0, \sqrt{M} / 4) \cap \spt\mu \subseteq K_{M}$.
  Therefore we deduce from \eqref{eq:dual1} that
    \[
    \varepsilon = \int [f_{\varepsilon}(x) + g_{\varepsilon}(y_0) - c(x, y_0) ]_+ \,\mu(\rmd x) \ge M \mu(K_{M}) \ge M  \mu\big( B(y_0, \sqrt{M}/4)) \big) \ge M \rho(\sqrt{M}/4),
  \]
  which implies $M \le  16\delta_{\mathsf{ST}}(\epsilon/16) \le 16\delta_{\mathsf{ST}}(\epsilon)$. Together with  Lemma~\ref{lem: sym ub by M} this concludes the proof.
\end{proof}

\subsection{Proof of the lower bound}

\begin{proof}[Proof of Theorem~\ref{thm: sym lb}]
  As in the proof of Theorem~\ref{thm: sym ub} we recall $M = \sup_{x \in \spt\mu} f_\epsilon(x)$, and choose $x_0 \in \spt \mu$ such that $f_\epsilon(x_0) = g_\epsilon(x_0) = M$.  
  We show that there exists $(x_0, y_0) \in \spt \pi_{\varepsilon}$ satisfying $\|x_0 - y_0\| \gtrsim \min\big( \sqrt{M}, \Diam(\spt\mu) \big)$: indeed, \eqref{eqn: large spread by grad ub} in the proof of Theorem~\ref{thm: sym ub} together with \eqref{eqn: def spt} readily implies that for all $y \in B(x_0, \sqrt{M} / 4) \cap \spt \mu$ we have $(x_0, y) \in \spt \pi_\epsilon$. We now distinguish two cases: 
  \begin{enumerate}[label=(\roman*)]
      \item If the intersection of $B(x_0, \sqrt{M} / 4) \setminus B(x_0, \sqrt{M} / 8)$ and $\spt \mu$ is non-empty, then any point $y_0$ in this intersection satisfies $(x_0, y_0) \in \spt\pi_\epsilon$ and $\|x_0 - y_0\| \ge \sqrt{M}/8$, as desired. 
      \item Now we assume that the intersection of $B(x_0, \sqrt{M} / 4) \setminus B(x_0, \sqrt{M} / 8)$ and $\spt \mu$ is empty. Since $x_0 \in \spt\mu$ and $\interior\spt\mu$ is connected, this implies $\spt\mu \subseteq B(x_0, \sqrt{M}/8)$, hence $B(x_0, \sqrt{M}/4) \cap \spt\mu = \spt\mu$. By \eqref{eqn: large spread by grad ub} we know that $(x_0, y) \in \spt\pi_\epsilon$ for all $y \in \spt\mu$. Using the triangle inequality and the definition of $\Diam(\spt\mu)$ there exists $y_0 \in \spt\mu$ such that $\|x_0 - y_0\| \ge \Diam(\spt\mu) / 2$. In conclusion $(x_0, y_0)\in \spt\pi_\epsilon$ fulfills the desired condition.
  \end{enumerate}
  
  Having established the existence of $(x_0, y_0) \in \spt\pi_\epsilon$ such that $\|x_0 - y_0\| \gtrsim \min\big( \sqrt{M}, \Diam(\spt\mu) \big)$, it remains to show $M \gtrsim \delta_{\mathsf{ST}}(\varepsilon)$. First, by Lemma~\ref{lem: sym ub by M} we have $\spt \pi_{\varepsilon} \subseteq \{(x, y) \in \spt \bbP: \|x - y\|^2 \le 4M\}.$
  Using this fact \eqref{eq:dual1} yields
  \[
  \varepsilon = \int \big[f_{\varepsilon}(x) + f_{\varepsilon}(y) - c(x, y)\big]_+ \,\mu(\rmd y) 
  \le 2M \cdot \mu\left( B(x, 2\sqrt{M}) \right) \qquad \forall x \in \spt\mu.
  \]
  By definition of $\delta_{\mathsf{ST}}$ we conclude that $M \ge \delta_{\mathsf{ST}}(2\epsilon) / 4 \ge \delta_{\mathsf{ST}}(\epsilon) / 4$, and the conclusion follows. 
\end{proof}

\section{Preparations for the general case: bounding the density $\rmd\pi_\epsilon/\rmd\bbP$}\label{sec:prep}

Now we turn to the general setting $\mu \ne \nu$. As in Section \ref{sec:self-transport} we first give an upper bound of the re-normalised density $$\epsilon\cdot \frac{\rmd \pi_\epsilon}{\rmd \bbP}(x,y)= f_\epsilon(x) + g_\epsilon(y) - c(x, y).$$ 
 To achieve this, we establish Lipschitz regularity of $f_\epsilon$ (Lemma \ref{lem:continuity}), and then use \eqref{eq:dual1} to give an upper bound on $\epsilon\cdot \rmd \pi_\epsilon/\rmd \bbP$ in terms of the $\epsilon$--spread of $\mu$ (Lemma \ref{lem: density upper bound}). In Lemma \ref{lem: approx conj} we show that, up to an error of order $\delta(\epsilon)$, the dual potentials $f_\epsilon$ and $g_\epsilon$ are conjugate. This observation relies on Lemma \ref{lem: density upper bound} and will be crucial for the proofs of Theorems \ref{thm: concentration}, \ref{thm: general ub} and \ref{thm: bdry ub} in Section \ref{sec:general}.

To start we recall the following result on the modulus of continuity of the dual potentials, which was proved in \cite[Lemma 2.5.(iii)]{nutz2024quadratically}.

\begin{lemma}\label{lem:continuity}
For all $y\in \R^d$ the functions $x\mapsto c(x,y)$ and  $x\mapsto f_{\varepsilon}(x)$ are $2$-Lipschitz continuous.
\end{lemma}

\begin{proof}
Note that $c(x,y)=\|x-y\|^2/2$ is $2$-Lipschitz on $B(0,1).$ The claim now follows from \cite[Lemma 2.5.(iii)]{nutz2024quadratically}.
\end{proof}

We are now in a position to give an upper bound on the density of $\pi_\epsilon.$
\begin{lemma}\label{lem: density upper bound}
We have
  \[
  \varepsilon \cdot \sup_{(x, y) \in \spt\bbP} \frac{\rmd \pi_{\varepsilon}}{\rmd \bbP} = \sup_{(x, y) \in \spt\bbP} \big[ f_{\varepsilon}(x) + g_{\varepsilon}(y) - c(x, y) \big] \le 5\delta(\varepsilon),
  \]
  where 
  \[
  \delta(\varepsilon) = \inf\{r>0 : r \cdot \rho(r) > \varepsilon\}, \quad 
  \rho(r) \coloneqq \inf_{x \in \spt\mu} \mu(B(x, r)).
  \]
\end{lemma}

\begin{proof}
  Suppose to the contrary that $f_{\varepsilon}(x_0) + g_{\varepsilon}(y_0) - c(x_0, y_0) > 5\delta(\varepsilon)$ for some $(x_0, y_0) \in \spt \bbP=\spt \mu\times \spt \nu$. Recall from Lemma \ref{lem:continuity} that $x\mapsto c(x,y_0)$ and $x\mapsto f_{\varepsilon}(x)$ are $2$-Lipschitz continuous. Therefore, we have
  \[
    f_{\varepsilon}(x) + g_{\varepsilon}(y_0) - c(x, y_0) > \delta(\varepsilon) \qquad \forall x \in B(x_0, \delta(\varepsilon)).
  \]
  Consequently, by \eqref{eq:dual1} and the definition of $\delta$ we have
  \[
    \varepsilon = \int \big[f_{\varepsilon}(x) + g_{\varepsilon}(y_0) - c(x, y_0) \big]_+ \,\mu(\rmd x) > \delta(\varepsilon) \cdot \mu\big( B(x_0, \delta(\varepsilon)) \big)\ge \epsilon,
  \]
  a contradiction. This concludes the proof.
\end{proof}

For completeness, we also list properties of the $\epsilon$--spread $\delta(\epsilon)$, which will  frequently be used in the proofs below. These again follow from the definition.

\begin{lemma}\label{lem:delta}
  We have the following:
  \begin{itemize}
      \item $\epsilon\mapsto \delta(\varepsilon)$ is monotone increasing.
      \item $\delta(\varepsilon) \ge \varepsilon$ for $\varepsilon>0$.
      \item $\delta(\varepsilon) = \varepsilon$ for $\varepsilon\ge 2$ since $\rho(r) = 1$ for $r \ge 2$. 
      \item $\delta(C\epsilon) \le \max(C, 1) \delta(\varepsilon)$.
  \end{itemize}
\end{lemma}

We are now in a position to derive the following upper bound on the support of $\pi_{\varepsilon}$.

\begin{lemma}[Dual potentials are approximately conjugate to each other]\label{lem: approx conj}

There exists a continuously differentiable, $1$-Lipschitz convex function  $\psi_{\varepsilon}: \bbR^d \to \bbR$ satisfying 
  \begin{align}
    \left| \frac12\|x\|^2 - f_{\varepsilon}(x) - \psi_{\varepsilon}(x) \right| \le 6\delta(\varepsilon) \qquad & \forall x \in \spt \mu\label{eq:toshow1},
    \\
    \left| \frac12\|y\|^2 - g_{\varepsilon}(y) - \psi_{\varepsilon}^*(y) \right| \le 6\delta(\varepsilon) \qquad & \forall y \in \spt \nu. 
    \label{eq:toshow2}
  \end{align}
  Consequently we have
  \begin{equation}
    \label{eqn: ub by psi}
    \spt \pi_{\varepsilon} \subseteq \{(x, y)\in \spt\bbP: \psi_{\varepsilon}(x) + \psi_{\varepsilon}^*(y) - \langle x, y \rangle < 12\delta(\varepsilon)\}. 
  \end{equation}
\end{lemma}

\begin{proof}
For all $x\in \R^d$ set 
    \begin{align*}
    \widetilde{\psi}_{\varepsilon}(x) \coloneqq \sup_{y \in \spt\nu} \Big[ \langle x, y \rangle - \big( \frac12 \|y\|^2 - g_{\varepsilon}(y) \big) \Big].
    \end{align*}
    Note that $x\mapsto \widetilde{\psi}_{\varepsilon}(x)$ is convex and lower semicontinuous, and is $1$-Lipschitz since its gradient lies in $\spt \nu \subseteq B(0,1)$.
    Now define the Moreau envelope
    \begin{align*}
    \psi_{\varepsilon}(x) := \inf_{y\in \R^d} \Big[\widetilde{\psi}_{\varepsilon}(y) + \frac{1}{2\lambda} \|x-y\|^2\Big],
    \end{align*}
    where $\lambda>0$ is chosen such that $|\widetilde{\psi}_{\varepsilon}(x)-\psi_{\varepsilon}(x)|\le \delta(\epsilon)$ for all $x \in \bbR^d$. 
  By properties of Moreau envelopes we know that $\psi_{\varepsilon}:\R^d\to \R$ is convex, continuously differentiable, and $1$-Lipschitz. 
  It follows from Lemma~\ref{lem: density upper bound} and \eqref{eqn: density lower bound y} that
  \begin{align}\label{eq:key_estimate}
  -5\delta(\varepsilon) \le \frac12 \|x\|^2 - f_{\varepsilon}(x) - \widetilde \psi_{\varepsilon}(x) = -\sup_{y\in \spt\nu} [f_\varepsilon(x)+g_\varepsilon(y)-c(x,y)] \le -\varepsilon.
  \end{align}
  This yields \eqref{eq:toshow1}, recalling that $|\widetilde{\psi}_{\varepsilon}(x)-\psi_{\varepsilon}(x)| \le \delta(\epsilon)$.
  The proof of equation \eqref{eq:toshow2} is slightly more involved, but follows essentially from the same arguments. Let us first recall that $\widetilde\psi_{\varepsilon}^*$ is the closed convex hull of $y \mapsto \|y\|^2/2 - g_{\varepsilon}(y)$; in other words, $\widetilde\psi_{\varepsilon}^*(y) \le \|y\|^2/2 - g_{\varepsilon}(y)$ for all $y \in \spt \nu$, and for any lower-semicontinuous convex function $\xi: \bbR^d \to \bbR$ that satisfies $\xi(y) \le \frac12 \|y\|^2 - g_{\varepsilon}(y)$ for all $y \in \spt\nu$, we have $\widetilde\psi_{\varepsilon}^* \ge \xi$. Now we consider
  \begin{equation}
    \label{eqn: def eta}
  \eta(y) \coloneqq \sup_{x \in \spt\mu} \Big[ \langle x, y \rangle - \big( \frac12 \|x\|^2 - f_{\varepsilon}(x) \big) \Big] - 5\delta(\varepsilon).
  \end{equation}
  Then $\eta$ is lower-semicontinuous and convex, and it again follows from Lemma~\ref{lem: density upper bound} and \eqref{eqn: density lower bound x} using the same arguments as in \eqref{eq:key_estimate} that
  \begin{equation}
    \label{eqn: eta and potential}
  0 \le \frac12 \|y\|^2 - g_{\varepsilon}(y) - \eta(y) \le 5\delta(\varepsilon) -\varepsilon
  \qquad\forall y \in \spt \nu.
  \end{equation}
  From this we conclude that $\eta(y) \le \frac12 \|y\|^2 - g_{\varepsilon}(y)$ on $\spt \nu$, and thus $\widetilde\psi_\epsilon^* \ge \eta$. Consequently
  \begin{equation*}
  \frac12 \|y\|^2 - g_{\varepsilon}(y) + \varepsilon - 5\delta(\varepsilon)\le \eta(y)\le  \widetilde\psi_{\varepsilon}^*(y).
  \end{equation*}
  Combined with the inequality $\widetilde\psi_{\varepsilon}^*(y) \le \frac12 \|y\|^2 - g_{\varepsilon}(y)$ for all $y \in \spt \nu$, we obtain
  \begin{align*}
   \frac12 \|y\|^2 - g_{\varepsilon}(y) + \varepsilon - 5\delta(\varepsilon)\le \widetilde\psi_{\varepsilon}^*(y) \le \frac12 \|y\|^2 - g_{\varepsilon}(y)
  \end{align*}
  and so we conclude
  \[
    0 \le \frac12 \|y\|^2 - g_{\varepsilon}(y) - \widetilde\psi_{\varepsilon}^*(y) \le 5\delta(\varepsilon) - \varepsilon.
  \]
  This yields~\eqref{eq:toshow2}, recalling that $|\widetilde\psi_\epsilon(x)- \psi_\epsilon(x)| \le \delta(\epsilon)$ for all $x\in \R^d$ and hence $|\widetilde\psi_\epsilon^*(y) - \psi_\epsilon^*(y)| \le \delta(\epsilon)$ whenever $\widetilde\psi_\epsilon^*(y)$ is finite.
  Lastly we use \eqref{eq:toshow1} and \eqref{eq:toshow2} to see that
  \begin{align*}
  \spt \pi_\epsilon = \overline{\left\{ (x, y) \in \spt \bbP: f_{\varepsilon}(x) + g_{\varepsilon}(y) > c(x,y) \right\}} \subseteq \{(x, y)\in \spt\bbP: \psi_{\varepsilon}(x) + \psi_{\varepsilon}^*(y) - \langle x, y \rangle < 12\delta(\varepsilon)\}.    
  \end{align*}
  This completes the proof.
\end{proof}

\begin{cor}\label{cor:bound}
Define 
\begin{align}\label{eq:surrogate}
\psi_{\varepsilon}'(y) := \sup_{x \in \spt \mu} [ \langle x, y \rangle - \psi_{\varepsilon}(x) ]
\end{align}
for $y\in \R^d.$
We have
\begin{equation}
\label{eqn: psi dual restricted}
\big| \psi_{\varepsilon}^*(y) - \psi_{\varepsilon}'(y)  \big| \le 22 \delta(\varepsilon)
\qquad \forall y \in \spt\nu.
\end{equation}
\end{cor}

\begin{proof}
From \eqref{eq:toshow1} we have
\begin{align*}
& \phantom{{}={}}\Big| \psi_\epsilon'(y) - \sup_{x \in \spt \mu} \big[ \langle x, y\rangle - \big( \frac12 \|x\|^2 - f_\epsilon(x) \big) \big] \Big|
\\
& = \Big|\! \sup_{x \in \spt \mu} [ \langle x, y\rangle - \psi_\epsilon(x) ] - \sup_{x \in \spt \mu} \big[ \langle x, y\rangle - \big( \frac12 \|x\|^2 - f_\epsilon(x) \big) \big] \Big|
\\
&\le \sup_{x \in \spt\mu} \left| \frac12\|x\|^2 - f_{\varepsilon}(x) - \psi_{\varepsilon}(x) \right|
\\
&\le 6 \delta(\epsilon).
\end{align*}
Combining this with the definition~\eqref{eqn: def eta} of $\eta$ in the proof of Lemma~\ref{lem: approx conj}, we see that
\begin{align*}
|\psi_\epsilon'(y) - \eta(y)|
&\le \Big| \psi_\epsilon'(y) - \sup_{x \in \spt \mu} \big[ \langle x, y\rangle - \big( \frac12 \|x\|^2 - f_\epsilon(x) \big) \big] \Big| + \Big|\sup_{x \in \spt \mu} \big[ \langle x, y\rangle - \big( \frac12 \|x\|^2 - f_\epsilon(x) \big) \big] - \eta(y)\Big|\\
&\le 6\delta(\epsilon)+5\delta(\epsilon)=11\delta(\epsilon).
\end{align*}
Now using the triangle inequality and \eqref{eqn: eta and potential} in the proof of Lemma~\ref{lem: approx conj} we obtain
\[
\Big| \psi_\epsilon'(y) - \big( \frac12 \|y\|^2 - g_\epsilon(y) \big) \Big| \le |\psi_\epsilon'(y) - \eta(y)| + \Big| \eta(y) - \big( \frac12 \|y\|^2 - g_\epsilon(y) \big) \Big| \le 11\delta(\epsilon)+5 \delta(\epsilon)=16\delta(\epsilon).
\]
The desired bound on $|\psi_\epsilon'(y) - \psi_\epsilon^*(y)|$ follows readily from this and~\eqref{eq:toshow2}.
\end{proof}

\section{General setting}\label{sec:general}

\subsection{Main results}

Now we turn to the general setting $\mu \ne \nu$. As mentioned in the Introduction we establish a concentration bound, which shows that $\spt \pi_\epsilon$ is close to the graph of $\nabla\psi_\epsilon$ for the convex function $\psi_\epsilon$ (Theorem \ref{thm: concentration}), and a bias bound, which states that the graphs of $\nabla\psi_\epsilon$ and $\nabla\varphi$ are close to each other whenever $\nabla\varphi$ is Lipschitz (Theorems \ref{thm: general ub}, \ref{thm: bdry ub}).

If we assume that $\mu$ has a density bounded from below and $\spt \mu$ has Lipschitz boundary, then the concentration bound in Theorem \ref{thm: concentration} is of order $\epsilon^{\frac{1}{2+2d}}$, which is only polynomially worse than the optimal rate $\epsilon^{\frac{1}{2+d}}$ in Section \ref{sec:self-transport} (see Corollary~\ref{rem:rate}). On the other hand, the bias bound in Theorem \ref{thm: general ub} is only of order $\epsilon^{\frac{1}{(2(1+d))^2}}$. We suspect that the optimal rate is again $\epsilon^{\frac{1}{2+d}}$ as in Section \ref{sec:self-transport}, which  seems out of reach with the tools currently available. We leave an improvement of our rates for future research.

\begin{theorem}[Concentration bound]
\label{thm: concentration}
There exists a convex continuously differentiable function $\psi_\epsilon: \bbR^d \to \bbR$, such that $\spt\pi_\epsilon$ concentrates around $\gr \nabla\psi_\epsilon = \{(x, \nabla\psi_\epsilon(x)): x \in \bbR^d\}$, the graph of $\nabla\psi_\epsilon$, i.e.,
\[\dist(\spt \pi_\epsilon; \gr \nabla\psi_\epsilon) \le C \sqrt{\delta(\epsilon)}\] for some absolute constant $C>0$.
\end{theorem}

\begin{theorem}[Bias bound]
\label{thm: general ub}
Assume that $\nabla \varphi$ is $L$-Lipschitz. Then there exists a universal constant $C > 0$ such that for all $(x, y) \in \spt \pi_{\varepsilon}$ we have
\[
\| y - \nabla \varphi(x) \| \le r \quad \text{ or }\quad  d(x, \partial \spt \mu) \le r, \text{where }\quad r \coloneqq C  (L+1)^{3/2} \sqrt{\delta\Big(\frac{\delta(\varepsilon)}{L+1}\Big)}.
\]
\end{theorem}

The bounds in Theorem \ref{thm: general ub} do not yield control of $\spt\pi_\epsilon$ within $r$ distance of the boundary of $\spt\mu$. It can be dropped under a mild regularity assumption on  $\spt\mu$. This is the content of our next result.

\begin{theorem}[Improved bias bound near boundary]
\label{thm: bdry ub}
Assume $\nabla \varphi$ is $L$-Lipschitz and that  $\spt \mu$ is star-shaped. Then for any $(x, y) \in \spt \pi_\varepsilon$ we have
\[
\| y - \nabla \varphi(x) \| \le C (L+1)^{3/2} \max\left( 
\sqrt[4]{\delta\Big(\frac{\delta(\varepsilon)}{L+1}\Big)}, \sqrt{\delta\Big(\frac{\delta(\varepsilon)}{L+1}\Big)} \right).
\]
\end{theorem}

Recall that a set $A\subseteq \R^d \times \R^d$ is star-shaped, if there exists a point $(x_0,y_0)\in A$ such that for any $(x,y)$ the inclusion
\begin{align*}
\{t(x_0,y_0)  + (1-t) (x,y): t\in [0,1]\}\subseteq A
\end{align*}
holds.
The proof of Theorem \ref{thm: bdry ub} only depends on the shape of $\spt\mu$ through the constant $C>0. $ In fact we obtain the following corollary:

\begin{cor}\label{Cor: bdr ub}
Assume $\nabla \varphi$ is $L$-Lipschitz. Then for any $(x, y) \in \spt \pi_\varepsilon$ we have
\[
\| y - \nabla \varphi(x) \| \le C\sqrt{\ell(\mu)} (L+1)^{3/2} \max\left( 
\sqrt[4]{\delta\Big(\frac{\delta(\varepsilon)}{L+1}\Big)}, \sqrt{\delta\Big(\frac{\delta(\varepsilon)}{L+1}\Big)} \right),
\]    
where $$\ell(\mu):= \sup\{ \text{length of shortest path in $\spt \mu$ between }x, x' : x,x'\in \spt \mu\}. $$
\end{cor}
By assuming $\spt\mu$ is star-shaped, we essentially enforce $\ell(\mu) \le 4$.

The bound in Theorem~\ref{thm: bdry ub} is quadratically worse than the one in Theorem~\ref{thm: general ub} (at least if one ignores the dependence on $L$). However, one can always combine Theorem~\ref{thm: general ub} and Theorem~\ref{thm: bdry ub}. In that sense, Theorem~\ref{thm: bdry ub} should be viewed as an improvement of Theorem~\ref{thm: general ub} near the boundary of $\spt\mu$. In fact, the following more general result follows from our derivations: if there exists a uniform bound  $\psi_{\varepsilon}(x) + \psi_{\varepsilon}^*(\nabla \varphi(x)) - \langle x, \nabla\varphi(x) \rangle \le \alpha(\varepsilon)$ or, in terms of the dual potentials, $f_\epsilon(x) + g_\epsilon(\nabla\varphi(x)) - c(x, \nabla\varphi(x)) \ge -\alpha(\epsilon)$, for some function $\alpha: \bbR_+ \to \bbR_+$, then following the  proof of Theorem~\ref{thm: general ub} produces an upper bound of order $\sqrt{\alpha(\varepsilon) + \delta(\epsilon)}$, while following the proof of Theorem~\ref{thm: bdry ub} produces an upper bound of order $\sqrt[4]{\alpha(\varepsilon) + \delta(\epsilon)}$.

\subsection{Proof of the concentration bound}
The proof of Theorem~\ref{thm: concentration} combines Lemma~\ref{lem: approx conj} and the following quadratic detachment lemma derived from Minty's trick \cite{minty1962monotone}.

\begin{lemma}
\label{lem: minty}
Let $\phi: \bbR^d \to \bbR$ be a continuously differentiable convex function. Then the following hold:
\begin{enumerate}[label=(\roman*)]
    \item For any $u \in \bbR^d$ there exists $x \in \bbR^d$ such that $x + \nabla\phi(x) = u$.
    \item There exists a Lipschitz function $F_{\phi}: \bbR^d \to \bbR^d$ such that 
    \[
    F_{\phi} \big( x + \nabla\phi (x) \big) = x - \nabla\phi(x).
    \]
    As a consequence, if $x$ solves $x + \nabla\phi(x) = u$ for a given $u\in \R^d$, then
    \[
    x = \frac12 (u + F_\phi(u)).
    \]
    \item 
    For any $(x, y) \in \bbR^d \times \bbR^d$ we have
    \[
    \phi(x) + \phi^*(y) - \langle x, y\rangle \ge \frac14 \left\| x - y - F_{\phi}(x + y) \right\|^2.
    \]
    Alternatively, let $x' \in \bbR^d$ be such that $x' + \nabla \phi(x') = x + y$. Then we have
    \begin{align}\label{eq:key}
    \phi(x) + \phi^*(y) - \langle x, y\rangle \ge \|x - x'\|^2 = \| y - \nabla \phi(x') \|^2.
    \end{align}
\end{enumerate}
\end{lemma}

\begin{proof}
Claim (i) was proved in \cite[Corollary after Theorem 4]{minty1962monotone}. Claim (ii) was essentially proved within the same reference, but for completeness we produce a short proof of (ii) here taking (i) for granted.\footnote{In \cite{minty1962monotone}, (i) was actually proved as a corollary of a more complicated version of (ii).} For any $u \in \bbR^d$, choose a solution of the equation $x(u) + \nabla \phi(x(u)) = u$ and call it $x(u)$. Such a solution $x(u)$ must exist by virtue of (i), but might not be unique at this point. Set
\[
F_\phi(u) \coloneqq x(u) - \nabla\phi(x(u)).
\]
We now show that $F_\phi$ is indeed a Lipschitz function. In fact, for any $u, v \in \bbR^d$, we have
\begin{align}
\| F_\phi(u) - F_\phi(v) \|^2 
& = \left\| x(u) - \nabla\phi(x(u)) - \big( x(v) - \nabla\phi(x(v)) \big) \right\|^2
\nonumber\\
& = \left\| x(u) + \nabla\phi(x(u)) - \big( x(v) + \nabla\phi(x(v)) \big) + 2[\nabla\phi(x(v)) - \nabla\phi(x(u))] 
 \right\|^2
\nonumber\\
& = \left\| x(u) + \nabla\phi(x(u))  - \big( x(v) + \nabla\phi(x(v)) \big) \right\|^2 - 4 \langle x(u) - x(v), \nabla \phi(x(u) - \nabla \phi(x(v)) \rangle 
\nonumber\\
& \le \left\| x(u) + \nabla\phi(x(u)) - \big( x(v) + \nabla\phi(x(v)) \big) \right\|^2 
\nonumber\\
& = \| u - v \|^2,
\label{eq:lipschitz}
\end{align}
where the third equality follows from expanding the square and cancelling terms, the inequality follows from gradient monotonicity of convex functions, and the last equality follows from the definition of $x(u)$ and $ x(v)$. To establish (ii) it only remains to show that $x(u)$ is uniquely determined by $u$. But this follows again from \eqref{eq:lipschitz}: if 
\begin{align}\label{eq:unique}
 x(u) + \nabla \phi(x(u)) = x'(u) + \nabla \phi(x'(u)) = u,   
\end{align}
then \eqref{eq:lipschitz} (with $x(v)$ replaced by $x'(u)$) implies that 
\[\| x(u) - \nabla \phi(x(u)) - (x'(u) - \nabla \phi(x'(u))) \|^2 \le 0,\] 
thus $x(u) - \nabla \phi(x(u)) = x'(u) - \nabla \phi(x'(u))$. Adding this to \eqref{eq:unique} we obtain $2x(u) = 2x'(u)$, hence $x(u) = x'(u)$ as desired. Lastly we assume that $x$ solves $x+\nabla\phi(x)=u$ and compute 
\begin{align*}
F_\phi(u)=F_\phi(x+\nabla\phi(x))=x-\nabla \phi(x)= 2x-(x+\nabla \phi(x)) =2x- u.
\end{align*}
This yields $x=(u+F_\phi(u))/2,$ as desired.
For (iii) observe that for any $x' \in \bbR^d$ we have
\begin{align}\label{eq:tidy}
\begin{split}
\phi(x) + \phi^*(y) - \langle x, y \rangle 
& \ge \phi(x) + [\langle x', y \rangle - \phi(x') ] - \langle x, y \rangle 
\\
& = \phi(x) - \phi(x') + \langle x' - x, y \rangle 
\\
& \ge \langle \nabla \phi(x'), x - x' \rangle + \langle x' - x, y \rangle 
\\
& = -\langle x - x', y - \nabla\phi(x') \rangle.
\end{split}
\end{align}
Let $x' \in \bbR^d$ be such that $x' + \nabla \phi(x') = x+y$. Applying (ii) we obtain
\[
x' = \frac12 (x+y + F_\phi(x+y)).
\]
Then
\[
\nabla\phi(x') = x+y - x' = \frac12 (x+y - F_\phi(x+y)).
\]
Plugging these expressions into  $x'$ and $\nabla\phi(x')$ we can rewrite the last line in \eqref{eq:tidy} as
\begin{align*}
-\langle x - x', y - \nabla\phi(x') \rangle = \langle \frac{1}{2} (F_\phi(x+y)+y-x), \frac{1}{2}(y-x +F_\phi(x+y)\rangle = \frac{1}{4} \|x-y-F_\phi(x+y)\|^2,    
\end{align*}
as desired. Lastly \eqref{eq:key} follows from using $x'=(x+y+F_\phi(x+y))/2$, which holds by (ii), and the fact that $x'-x=y-\nabla \phi(x')$ by assumption. This concludes the proof.
\end{proof}

We are now ready to prove Theorem~\ref{thm: concentration}.
\begin{proof}[Proof of Theorem~\ref{thm: concentration}]
By Lemma~\ref{lem: approx conj} there exists a convex continuously differentiable function $\psi_\epsilon: \bbR^d \to \bbR$ satisfying
\[
\spt \pi_{\varepsilon} \subseteq \{(x, y)\in \spt\bbP: \psi_{\varepsilon}(x) + \psi_{\varepsilon}^*(y) - \langle x, y \rangle < 12\delta(\varepsilon)\}. 
\]
By Lemma~\ref{lem: minty}.(i)$\&$(iii) applied to $\psi_\epsilon$ we obtain the following: for any $(x, y) \in \spt\pi_\epsilon$ there exists $x' \in \bbR^d$ which solves $x' + \nabla\psi_\epsilon(x') = x + y$ and satisfies
\[
12 \delta(\varepsilon) \ge \psi_\epsilon(x) + \psi_\epsilon^*(y) - \langle x, y \rangle \ge \|x - x'\|^2 = \|y - \nabla\psi_\epsilon(x')\|^2.
\]
As a consequence, the distance between $(x, y)$ and $(x', \nabla\psi_\epsilon(x')) \in \gr \nabla\psi_\epsilon$ is no larger than \[\sqrt{\|x - x'\|^2 + \|y - \nabla\psi_\epsilon(x')\|} \le \sqrt{24 \delta(\epsilon)}.\] 
Since  $(x, y) \in \spt\pi_\epsilon$ were arbitrary, we obtain that
\[
\dist(\spt\pi_\epsilon; \gr\nabla\psi_\epsilon) \le \sqrt{24\delta(\epsilon)}.
\]
This completes the proof.
\end{proof}

\subsection{Proof of the bias bound}

Given the bound on $\dist(\spt\pi_\epsilon; \gr\nabla\psi_\epsilon)$ stated in  Theorem~\ref{thm: concentration}, our plan is to prove Theorem~\ref{thm: general ub} by quantifying how much $\gr\nabla \psi_\epsilon$ deviates from $\gr\nabla \varphi$. To achieve this, we first derive integral bounds of the gap between $ \psi_\epsilon$ and $\varphi$ (Lemma \ref{lem:1}, Corollary \ref{cor: gap}), which  we then turn into a pointwise bound in terms of the $\epsilon$--spread of $\mu$, using the same argument as in the proof of Lemma \ref{lem: density upper bound}  (Lemma \ref{lem: discrepancy ub}). In a last step we use Minty's trick to convert this pointwise point between $ \psi_\epsilon$ and $\varphi$ into a bound on the distance of the graphs of $\nabla \psi_\epsilon$ and $\nabla \varphi$ (Lemma \ref{lem: discrepancy ub to support ub}).

Our first result relates the integrals of the dual optimizers $f_{\varepsilon}, g_{\varepsilon}$ and the Kantorovich potentials $\varphi, \varphi^*$ via the upper bound on the density $\rmd \pi_\varepsilon/ \rmd\mathbb P$ derived in Lemma~\ref{lem: density upper bound}.

\begin{lemma}\label{lem:1}
  Recall $$C(\mu, \nu) = \int c\,\rmd\pi_\star = \int \varphi \, \rmd \mu + \int \varphi^* \rmd \nu.$$ The primal and the dual optimizer for (\QOT{}) satisfy
  \[
  C(\mu, \nu) \le \int c \, \rmd \pi_{\varepsilon} \le \int f_{\varepsilon}\, \rmd\mu + \int g_{\varepsilon}\, \rmd\nu \le C(\mu, \nu) + 5\delta(\varepsilon).
  \]
\end{lemma}
\begin{proof}
  The inequality $C(\mu,\nu) = \int c\,\rmd\pi_\star \le \int c\,\rmd \pi_{\varepsilon}$ is obvious by optimality of $\pi_\star$. To prove the remaining inequalities,
  we begin with the observation
  \begin{align*}
  0\le \left\| \frac{\rmd \pi_{\varepsilon}}{\rmd \bbP} \right\|_{L^2(\bbP)}^2 
  &= \int \frac{1}{\varepsilon^2}[f_{\varepsilon}(x) + g_{\varepsilon}(y) - c(x,y)]_+^2 \, \bbP(\rmd x, \rmd y)\\
  &= \int \frac{1}{\varepsilon^2}[f_{\varepsilon}(x) + g_{\varepsilon}(y) - c(x,y)][f_{\varepsilon}(x)+ g_{\varepsilon}(y) - c(x,y)]_+\,  \bbP(\rmd x, \rmd y)\\
  &= \frac{1}{\varepsilon} \int [f_{\varepsilon}(x) + g_{\varepsilon}(y) - c(x,y)] \, \pi_{\varepsilon}(\rmd x, \rmd y).
\end{align*}
  As a consequence,
  \[\int f_{\varepsilon} \,\rmd \mu + \int g_{\varepsilon}\, \rmd \nu \ge \int c \, \rmd \pi_{\varepsilon}. \]
  To prove the last inequality, we apply Lemma~\ref{lem: density upper bound} to see that $f_{\varepsilon}(x) + g_{\varepsilon}(y) - c(x,y) \le 5\delta(\varepsilon)$. Integrating with respect to $\pi_\star$ we obtain
  \[\int f_{\varepsilon} \,\rmd\mu + \int g_{\varepsilon} \,\rmd\nu \le \int c\,\rmd\pi_\star + 5\delta(\varepsilon).
  \]
  This completes the proof.
\end{proof}

\begin{cor}
  \label{cor: gap}
  The convex function $\psi_{\varepsilon}$ from Lemma~\ref{lem: approx conj} fulfills
  \[
  \int \big[\psi_{\varepsilon}(x) + \psi_{\varepsilon}^*(\nabla\varphi(x)) - \langle x, \nabla\varphi(x) \rangle \big]\, \mu(\rmd x) \le 12\delta(\varepsilon).
  \]
\end{cor}

\begin{proof}
   By Lemma \ref{lem: approx conj} we have
   \begin{align*}
   &\int \big[\psi_{\varepsilon}(x) + \psi_{\varepsilon}^*(\nabla\varphi(x)) - \langle x, \nabla\varphi(x) \rangle \big]\, \mu(\rmd x)\\
   &\le \int \Big[\frac{1}{2} \|x\|^2 -f_\varepsilon(x) + \frac{1}{2} \|\nabla \varphi(x)\|^2 -g_\varepsilon(\nabla \varphi(x)) - \langle x, \nabla\varphi(x) \rangle \Big] \,\mu(\rmd x) +12\delta(\varepsilon).
   \end{align*}
   On the other hand, by Lemma \ref{lem:1} we conclude
   \begin{align*}
    &\int \Big[\frac{1}{2} \|x\|^2 - f_\varepsilon(x) + \frac{1}{2} \|\nabla \varphi(x)\|^2 -g_\varepsilon(\nabla \varphi(x)) - \langle x, \nabla\varphi(x) \rangle \Big]\, \mu(\rmd x) +12\delta(\varepsilon)\\
   &= \int \big[c(x, \nabla \varphi(x))  -f_\varepsilon(x) -g_\varepsilon(\nabla \varphi(x))\big] \, \mu(\rmd x)+12\delta(\varepsilon)\\
   &= \int \big[c(x, y)  -f_\varepsilon(x) -g_\varepsilon(y)\big] \, \pi_\star (\rmd x, \rmd y)+12\delta(\varepsilon)\\
   &\le 12\delta(\varepsilon).
   \end{align*}
\end{proof}

If $\nabla\varphi$ is $L$-Lipschitz, we can convert the integral bound given in Corollary~\ref{cor: gap} into a uniform bound using the same idea as in the proof of Lemma~\ref{lem: density upper bound}.

\begin{lemma} \label{lem: discrepancy ub}
Assume that $\nabla \varphi$ is $L$-Lipschitz. Then we have
\begin{equation}
\label{eqn: discrepancy ub}
\psi_{\varepsilon}(x) + \psi_{\varepsilon}^*(\nabla\varphi(x)) - \langle x, \nabla\varphi(x) \rangle \le C(L+1) \delta\!\left(\frac{\delta(\varepsilon)}{L+1}\right)
\qquad \forall x \in \spt\mu.
\end{equation}
\end{lemma}

\begin{proof}
Recall the function $\psi_{\varepsilon}': \bbR^d \to \bbR$ defined in \eqref{eq:surrogate} in Corollary~\ref{cor:bound} as
\[
\psi_{\varepsilon}'(y) = \sup_{x \in \spt \mu} \left[ \langle x, y \rangle - \psi_{\varepsilon}(x) \right].
\]
By Corollary~\ref{cor:bound} and Corollary~\ref{cor: gap} we obtain
\begin{align}\label{eq:track}
\begin{split}
&\int \left[ \psi_{\varepsilon}(x) + \psi_{\varepsilon}'(\nabla\varphi(x)) - \langle x, \nabla\varphi(x) \rangle \right]\, \mu(\rmd x) \\
&\le \int \big[\psi_{\varepsilon}(x) + \psi_{\varepsilon}^*(\nabla\varphi(x)) - \langle x, \nabla\varphi(x) \rangle \big]\, \mu(\rmd x) + \sup_{y\in \spt\nu} | \psi_{\varepsilon}^*(y) -\psi_\epsilon'(y) |\
\\
&\lesssim \delta(\varepsilon).
\end{split}
\end{align}
The integrand $ x\mapsto \psi_{\varepsilon}(x) + \psi_{\varepsilon}'(\nabla\varphi(x)) - \langle x, \nabla\varphi(x) \rangle $ is nonnegative on $\spt \mu$ by construction of $\psi_{\varepsilon}'$. We claim that it is also $2(L+1)$-Lipschitz. To see this, note that $\psi_{\varepsilon}$ is $1$-Lipschitz by Lemma~\ref{lem: approx conj}, and by definition of $\psi_{\varepsilon}'$, we have $\partial \psi_{\varepsilon}' \subseteq \spt \mu \subseteq B(0,1)$, hence $\psi_\epsilon'$ is $1$-Lipschitz. The claimed Lipschitz continuity follows from these inequalities, the chain rule, and the assumption that $\nabla\varphi$ is $L$-Lipschitz. 

Now we follow the same idea as in the proof of Lemma~\ref{lem: density upper bound} and Theorem~\ref{thm: sym lb}.
For this we define 
\[
M := \sup_{x \in \spt\mu} \left[ \psi_{\varepsilon}(x) + \psi_{\varepsilon}'(\nabla\varphi(x)) - \langle x, \nabla\varphi(x) \rangle \right].
\]
By the $2(L+1)$-Lipschitz continuity of $ x\mapsto \psi_{\varepsilon}(x) + \psi_{\varepsilon}'(\nabla\varphi(x)) - \langle x, \nabla\varphi(x) \rangle $ there exists $x_0 \in \spt\mu$ such that
\[
\psi_{\varepsilon}(x) + \psi_{\varepsilon}'(\nabla\varphi(x)) - \langle x, \nabla\varphi(x) \rangle \ge \frac{M}{2} \qquad \forall x \in B\left( x_0, \frac{M}{4(L+1)} \right). 
\]
Together with \eqref{eq:track} this yields

\begin{align*}
C\delta(\varepsilon) \ge \int \left( \psi_{\varepsilon}(x) + \psi_{\varepsilon}'(\nabla\varphi(x)) - \langle x, \nabla\varphi(x) \rangle \right) \mu(\rmd x) &\ge \frac{M}{2} \mu\left( B \Big( x_0, \frac{M}{4(L+1)} \Big) \right) \\
&\ge \frac{M}{2} \rho\left( \frac{M}{4(L+1)}  \right)
\end{align*}
recalling the definition of $\rho$ stated in Lemma~\ref{lem: density upper bound}. The above can be rewritten more compactly as
\[
R \rho(R) \le \frac{C\delta(\varepsilon)}{L+1}, \qquad \text{where }R \coloneqq \frac{M}{4(L+1)}.
\]
Thus by definition of the $\epsilon$--spread $\delta(\epsilon)$ and Lemma~\ref{lem:delta} we find
\[
R \le  \delta\!\left(\frac{C\delta(\varepsilon)}{L+1}\right)\le C \delta\!\left(\frac{\delta(\varepsilon)}{L+1}\right).
\]
Plugging in the definition of $R$ then yields
\begin{equation}
\label{eqn: discrepancy ub}
M \le C(L+1) \delta\!\left( \frac{\delta(\varepsilon)}{L+1} \right).
\end{equation}
Lastly applying \eqref{eqn: psi dual restricted} once again, we obtain
\begin{align*}
\psi_{\varepsilon}(x) + \psi_{\varepsilon}^*(\nabla\varphi(x)) - \langle x, \nabla\varphi(x) \rangle &\le M + \sup_{y\in \spt\nu} | \psi_{\varepsilon}^*(y) -\psi_\epsilon'(y) |\\
&\le 
 C(L+1) \delta\!\left(\frac{\delta(\varepsilon)}{L+1}\right) + C\delta(\varepsilon) \le C(L+1) \delta\!\left(\frac{\delta(\varepsilon)}{L+1}\right)
\end{align*}
for all $x\in \spt\mu,$
where the last inequality follows from 
Lemma \ref{lem:delta} and the computation
\begin{align}\label{eq:epsilon_easy}
\delta(\epsilon)=\delta \Big( \frac{(L+1)\epsilon}{L+1}\Big) \le (L+1) \delta \Big( \frac{\epsilon}{L+1}\Big) \le (L+1) \delta \Big( \frac{\delta(\epsilon)}{L+1}\Big).
\end{align}
This concludes the proof.
\end{proof}

Next, we show how to derive an upper bound on the distance between $\spt \pi_{\varepsilon}$ and $\gr\nabla\varphi$ from a uniform upper bound on $\psi_{\varepsilon}(x) + \psi_{\varepsilon}^*(\nabla\varphi(x)) - \langle x, \nabla\varphi(x) \rangle$.

\begin{lemma}
\label{lem: discrepancy ub to support ub}
Define 
\begin{align}\label{eq:alpha}
\alpha(\varepsilon) \coloneqq \sup_{x \in \spt\mu} \big[ \psi_{\varepsilon}(x) + \psi_{\varepsilon}^*(\nabla\varphi(x)) - \langle x, \nabla\varphi(x) \rangle \big].   
\end{align}
Then there exists a universal constant $C > 0$ such that 
\[
\| y - \nabla \varphi(x) \| \le C L \sqrt{\alpha(\varepsilon) + \delta(\varepsilon)} \quad \text{ or }\quad  d(x, \partial \spt \mu) \le C \sqrt{\alpha(\varepsilon) + \delta(\varepsilon)}
\]
for all $(x, y) \in \spt \pi_{\varepsilon}$.
\end{lemma}

\begin{proof}
We divide the proof into two steps.

\noindent {\em Step I: bounding $\dist(\gr \nabla \varphi,\gr \nabla\psi_{\varepsilon})$.}  
Fix $x\in \R^d.$
By Lemma~\ref{lem: minty}.(i) applied to $\psi_{\varepsilon}$, there exists $x' \in \bbR^d$ such that
\begin{equation}
\label{eqn: proj psi}
x' + \nabla\psi_{\varepsilon}(x') = x + \nabla\varphi(x).
\end{equation}
Note that \eqref{eqn: proj psi} can be rewritten as
\begin{align}\label{eq:order2}
x - x' = - (\nabla\varphi(x) - \nabla\psi_{\varepsilon}(x')).
\end{align}
Combining this with Lemma~\ref{lem: minty}.(iii) (applied with $\phi=\psi_\varepsilon$ and $y=\nabla\phi(x)$) we obtain
\begin{align*}
\|x - x'\|^2 &= \|  \nabla\psi_{\varepsilon}(x')-\nabla\varphi(x)\|^2 \\
&\le \psi_\varepsilon(x)+ \psi_\varepsilon^*(\nabla \varphi(x)) -\langle x, \nabla \varphi(x)\rangle
& \text{(by \eqref{eq:key})} \\
&\le \alpha(\varepsilon).
& \text{(by \eqref{eq:alpha})}
\end{align*}
Using \eqref{eq:order2} again, the Euclidean distance between $(x, \nabla\varphi(x))$ and $(x', \nabla\psi_{\varepsilon}(x'))$ is bounded by
\begin{align}
\sqrt{\|x - x'\|^2 + \|\nabla\varphi(x) - \nabla\psi_{\varepsilon}(x')\|^2}
= \sqrt{2}\|x - x'\| 
\le \sqrt{2\alpha(\varepsilon)}.
\label{eqn: euclidean distance tmp}
\end{align}
In conclusion, for any $x \in \spt \mu$, there exists some point $x' \in \bbR^d$ such that $x +  \nabla\varphi(x) = x' + \nabla\psi_{\varepsilon}(x')$, and the Euclidean distance between $(x, \nabla\varphi(x))$ and $(x', \nabla\psi_{\varepsilon}(x'))$ is at most $\sqrt{2\alpha(\varepsilon)}$.

\vspace{1em}
\noindent {\em Step II: finding contradiction with convexity of $\psi_\varepsilon$.} 
Define 
\begin{align}\label{eq:choice_r}
r := 160(L+1)\sqrt{\alpha(\varepsilon) + \delta(\varepsilon)}.
\end{align}
It suffices to show that for any $(x, y) \in \spt \pi_{\varepsilon}$ with $d(x, \partial \spt \mu) > r$ we have $\| y - \nabla\varphi(x) \| \le r$.
Suppose to the contrary that 
\begin{align}\label{eq:contradiction}
\| y - \nabla\varphi(x) \| > r
\end{align}
for some $(x,y)\in \spt\pi_\epsilon$ with $d(x, \partial \spt \mu)>r.$ 
By Lemma~\ref{lem: minty}.(i) applied with $\phi= \psi_{\varepsilon}$, there exists $x'\in \mathbb{R}^d$ such that 
$x + y
= x' + \nabla\psi_{\varepsilon}(x')$. Consequently, $x - x' = - (y - \nabla \psi_{\varepsilon}(x'))$. 
Taking this into account, we may invoke Lemma~\ref{lem: minty}.(ii-iii) (applied with $\phi=\psi_{\varepsilon}$) and \eqref{eqn: ub by psi} in Lemma~\ref{lem: approx conj} to obtain (using a similar computation as in Step I)
\begin{align}\label{eqn: x_psi bound}
\begin{split}
\| x - x' \|^2 &= \| y - \nabla \psi_{\varepsilon}(x') \|^2 \\
&\le  \psi_\varepsilon(x)+ \psi_\varepsilon^*(y) -\langle x,y\rangle\\
&\le 12\delta(\varepsilon). 
\end{split}
\end{align}
Next we define
\begin{align}\label{eq: x'}
\vec{x} \coloneqq x + \frac{y - \nabla \varphi(x)}{8(L+1)\| y - \nabla \varphi(x)\|} r
\end{align}
and note that 
\begin{align}\label{eq:distance}
\|x-\vec{x}\|= \frac{r}{8(L+1)}<r. 
\end{align}
By the assumption that $d(x, \partial \spt\mu) > r$, $x \in \spt\mu$ and \eqref{eq:distance} we have $\vec{x} \in \interior \spt \mu$. By Step I there is some $\vec{x}'$ such that the Euclidean distance between $(\vec{x}, \nabla\varphi(\vec{x}))$ and $(\vec{x}', \nabla\psi_{\varepsilon}(\vec{x}'))$ is at most $\sqrt{2\alpha(\varepsilon)}$. We now show that this contradicts convexity of $\psi_{\varepsilon}$, in particular the monotonicity property
\begin{align}\label{eq:contradict}
\langle \vec{x}' - x', \nabla \psi_{\varepsilon}(\vec{x}') - \nabla \psi_{\varepsilon}(x') \rangle \ge 0.
\end{align}
For this we first observe that $$r = 160(L+1)\sqrt{\alpha(\varepsilon) + \delta(\varepsilon)}\ge 160 (L+1) \sqrt{\alpha(\varepsilon)}$$ by \eqref{eq:choice_r}. Thus
\begin{align*}
\| (\vec{x}' - x') - (\vec{x} - x) \| &\le \| \vec{x}' - \vec{x} \| + \| x - x' \| \\
&\le \sqrt{2\alpha(\epsilon)}+ \sqrt{12 \delta(\epsilon)}&\text{(by Step I}, \, \eqref{eqn: x_psi bound})\\
& \le \sqrt{24(\alpha(\varepsilon) + \delta(\varepsilon))} \\
&\le \frac{r}{32(L+1)} &\text{(by \eqref{eq:choice_r})}\\
&= \frac{\|\vec{x} - x\|}{4} &\text{(by \eqref{eq:distance})}
\end{align*}
and
\begin{align*}
&\| \big( \nabla \psi_{\varepsilon}(\vec{x}') - \nabla \psi_{\varepsilon}(x') \big) - (\nabla\varphi(x) - y) \|\\
&= \left\| - \nabla \psi_{\varepsilon}(x') +y +\nabla \psi_{\varepsilon}(\vec{x}') - \nabla\varphi(\vec{x}') +\nabla\varphi(\vec{x}') - \nabla\varphi(x) \right\|\\
& \le \| \nabla \psi_{\varepsilon}(x') - y \| + \| \nabla \psi_{\varepsilon}(\vec{x}') - \nabla\varphi(\vec{x}) \| + \| \nabla\varphi(\vec{x}) - \nabla\varphi(x) \|
\\
& \le 2\sqrt{\alpha(\varepsilon)} + L \| \vec{x} - x \| &\text{(by \eqref{eqn: x_psi bound},  Step I, $\nabla \varphi$  $L$-Lipschitz)}\\
& \le 2\sqrt{\alpha(\varepsilon)} + \frac{r}{8} &\text{(by \eqref{eq:distance})}
\\
& \le \frac{r}{4} &\text{(by \eqref{eq:choice_r})}\\
&\le \frac{\| y - \nabla \varphi(x)\|}{4}. &\text{(by }\eqref{eq:contradiction})
\end{align*}
Using the two last inequalities together with the Cauchy-Schwarz inequality, we finally obtain
\begin{align*}
& \langle \vec{x}' - x', \nabla \psi_{\varepsilon}(\vec{x}') - \nabla \psi_{\varepsilon}(x') \rangle 
\\
& \quad = \left\langle  (\vec{x}- x) + [(\vec{x}' - x') - (\vec{x} - x)], ~
(\nabla\varphi(x) - y) + \big[\big( \nabla \psi_{\varepsilon}(\vec{x}') - \nabla \psi_{\varepsilon}(x') \big) - (\nabla \varphi(x) - y)\big]\right\rangle 
\\
& \quad
\le \langle \vec{x} - x, \nabla\varphi(x) - y \rangle 
+ \frac{\|\vec{x} - x\|}{4} \|y - \nabla\varphi(x)\| + \|\vec{x} - x\| \frac{\|y - \nabla\varphi(x)\|}{4} + \frac{\|\vec{x} - x\|}{4} \frac{\|y - \nabla\varphi(x)\|}{4}
\\
& \quad = -\frac{7\| y - \nabla\varphi(x)\| r}{64(L+1)}
\\
& \quad < 0,
\end{align*}
where we have used $$\langle \vec{x} - x, \nabla\varphi(x) - y \rangle = \Big\langle \frac{y-\nabla \varphi (x)}{8(L+1) \|y-\nabla \varphi(x)\|} r, \nabla \varphi(x) -y\Big \rangle = - \frac{\|y-\nabla \varphi(x)\|}{8(L+1)}r  $$
and \eqref{eq:distance}
for the last equality. This leads to a contradiction, as claimed.
\end{proof}

With Lemma~\ref{lem: discrepancy ub to support ub} in hand we can finally give the proof of Theorem~\ref{thm: general ub}. 

\begin{proof}[Proof of Theorem~\ref{thm: general ub}]
Recall $\alpha(\varepsilon)$ from \eqref{eq:alpha}. By Lemma~\ref{lem: discrepancy ub} we have
\[
\alpha(\varepsilon) \le C(L+1) \delta\!\left(\frac{\delta(\varepsilon)}{L+1}\right).
\]
Plugging this and the bound $\delta(\varepsilon) \le C(L+1) \delta(\delta(\varepsilon)/(L+1))$, as derived in \eqref{eq:epsilon_easy}, into Lemma~\ref{lem: discrepancy ub to support ub} immediately yields the desired conclusion.
\end{proof}

\subsection{Boundary behavior for star-shaped support}

\begin{proof}[Proof of Theorem~\ref{thm: bdry ub}]
Recall $$\alpha(\varepsilon) = \sup_{x \in \spt\mu} [ \psi_{\varepsilon}(x) + \psi_{\varepsilon}^*(\nabla\varphi(x)) - \langle x, \nabla\varphi(x) \rangle ]$$ from  Lemma~\ref{lem: discrepancy ub to support ub}. By definition of the convex conjugate $\psi_\epsilon^*$ we have for any $\overline x, \underline{x} \in \spt\mu$
\begin{align}
\begin{split}
\psi_{\varepsilon}(\overline{x}) 
&= \langle \overline{x}, \nabla \varphi(\underline{x}) \rangle - (\langle \overline{x}, \nabla \varphi(\underline{x}) \rangle- \psi_{\varepsilon}(\overline{x}))\\
& \ge \langle \overline{x}, \nabla \varphi(\underline{x}) \rangle - \psi_{\varepsilon}^*(\nabla \varphi(\underline{x}))
\\
& = \psi_{\varepsilon}(\underline{x}) + \langle \nabla \varphi(\underline{x}), \overline{x} - \underline{x} \rangle - \big( \psi_{\varepsilon}(\underline{x}) + \psi_{\varepsilon}^*(\nabla \varphi(\underline{x})) - \langle \underline{x}, \nabla \varphi(\underline{x}) \rangle \big)
\\
& \ge \psi_{\varepsilon}(\underline{x}) + \langle \nabla \varphi(\underline{x}), \overline{x} - \underline{x} \rangle - \alpha(\varepsilon).
\end{split}
\label{eqn: approx grad}
\end{align}
Fix a point $x \in \spt \mu$.
Since $\spt\mu$ is star-shaped, there exists a positive integer $K \le \lceil 4/r \rceil$, such that for any $x' \in \spt \mu\subseteq B(0,1)$ and for any $r>0$ we can find points $x_0, x_1, \dots, x_K \in \spt \mu$ with $x_0=x', x_K = x$, $\|x_{k+1} - x_k\| \le r$ for all $k=0, 1, \cdots, K-1$. 
Applying~\eqref{eqn: approx grad} with $\overline{x}=x_{k+1}, \underline{x}=x_k$ we have
\[
\psi_{\varepsilon}(x_{k+1}) - \psi_{\varepsilon}(x_k) \ge \langle \nabla\varphi(x_k), x_{k+1} - x_k \rangle - \alpha(\varepsilon).
\]
Summing the above inequality for $k=0, 1, \cdots, K-1$, we obtain
\begin{equation}
\label{eqn: grad sum}
\psi_{\varepsilon}(x) - \psi_{\varepsilon}(x') \ge \sum_{k=0}^{K-1} \langle \nabla \varphi(x_k), x_{k+1} - x_k \rangle - K \alpha(\varepsilon).
\end{equation}
On the other hand, since $\nabla\varphi$ is $L$-Lipschitz, it follows from a first-order Taylor expansion that
\[
\langle \nabla \varphi(x_k), x_{k+1} - x_k \rangle 
\ge \varphi(x_{k+1}) - \varphi(x_k) - L\|x_{k+1} - x_k\|^2 \ge \varphi(x_{k+1}) - \varphi(x_k) - Lr^2,
\]
where we have used $\|x_{k+1} - x_k\| \le r$ for the last inequality. Plugging this into~\eqref{eqn: grad sum}, we obtain
\[
\psi_{\varepsilon}(x) - \psi_{\varepsilon}(x') \ge \varphi(x) - \varphi(x') - LKr^2 - K\alpha(\varepsilon).
\]
Since $x, x' \in \spt\mu$ were arbitrary, we  deduce that
\begin{align*}
\varphi(x) + \varphi^*(y) - \langle x, y \rangle 
&= \varphi(x) + \sup_{x'\in \spt\mu} [\langle x',y\rangle -\varphi(x')]  - \langle x, y \rangle \\
&\le \psi_{\varepsilon}(x) + \sup_{x'\in \spt\mu} [\langle x',y\rangle -\psi_\varepsilon(x')]  - \langle x, y \rangle + LKr^2 + K\alpha(\varepsilon)\\
&= \psi_{\varepsilon}(x) +\psi_{\varepsilon}'(y)  - \langle x, y \rangle + LKr^2 + K\alpha(\varepsilon)\\
& \le 
\psi_{\varepsilon}(x) + \psi_{\varepsilon}^*(y) - \langle x, y \rangle  + \sup_{y\in\spt\nu} |\psi_{\varepsilon}^*(y) - \psi_{\varepsilon}'(y)| + LKr^2 + K\alpha(\varepsilon)
\\
& \le \psi_{\varepsilon}(x) + \psi_{\varepsilon}^*(y) - \langle x, y \rangle + C\delta(\varepsilon) + LKr^2 + K\alpha(\varepsilon)
\end{align*}
holds for all $(x, y)\in \spt \pi_\epsilon,$
where we have used $\partial \varphi^*(y)\in \spt \mu$ for the first equality, and Corollary \ref{cor:bound} for the last inequality.
Recalling \eqref{eqn: ub by psi} in Lemma~\ref{lem: approx conj} we conclude that
\[
\varphi(x) + \varphi^*(y) - \langle x, y \rangle 
\le C\delta(\varepsilon) + LKr^2 + K\alpha(\varepsilon)\qquad  \forall (x, y) \in \spt\pi_\varepsilon.
\]
On the other hand we have
\begin{align*}
\varphi(x) + \varphi^*(y) - \langle x, y \rangle &\ge
 \langle \nabla \varphi(x), x\rangle - \varphi^*(\nabla \varphi(x))+ \varphi^*(y) - \langle x, y \rangle \\
&= \varphi^*(y) -  \varphi^*(\nabla \varphi(x)) - \langle y-\nabla \varphi(x), x\rangle \\
&\ge \frac{1}{2L} \|y - \nabla\varphi(x)\|^2,
\end{align*}
where the first inequality follows from $(\varphi^*)^*=\varphi$ and the second inequality follows from the fact that $\varphi^*$ is $1/L$-strongly convex and $x\in \partial \varphi^*(\nabla\varphi(x)).$
Combining these inequalities,
\[
\| y - \nabla\varphi(x) \| \le \sqrt{2L[C\delta(\varepsilon) + Kr^2 + K\alpha(\varepsilon)]}.
\]
To prove the theorem, it remains to set $r=\sqrt{\alpha(\varepsilon)}$, noting that
\[
Kr^2 + K\alpha(\varepsilon) =2K \alpha(\varepsilon) 
\le \begin{cases}
    C\sqrt{\alpha(\varepsilon)}, & \alpha(\varepsilon) \le 1, \\
    C\alpha(\varepsilon), & \alpha(\varepsilon) > 1,
\end{cases}
\]
since $K\le \lceil 4/r\rceil \le \begin{cases}
    8/r, & r \le 1, \\ 4, & r > 1.
\end{cases}$
\end{proof}

\begin{proof}[Proof of Corollary \ref{Cor: bdr ub}]
Note that every pair of points in $\spt \mu$ can be joined by a shortest path in $\spt\mu$ since $\interior\spt\mu$ is  connected, and $\ell(\mu)$ is finite since $\spt\mu$ is compact. Going through the proof of Theorem~\ref{thm: bdry ub} we only need to replace $K\le \lceil 4/r \rceil$ by $K \le \lceil \ell(\mu) / r\rceil$ and hence $C$ by $C\sqrt{\ell(\mu)}$.
\end{proof}

\bibliographystyle{plain}
\bibliography{refs}

\end{document}